\documentclass[11pt,reqno]{amsart}
\usepackage[T1]{fontenc}
\usepackage{lmodern}
\usepackage[margin=1in]{geometry}
\usepackage{amsmath,amssymb,amsthm,mathtools}
\usepackage{xcolor}
\newcommand{\reviewblue}[1]{#1}
\usepackage[expansion=false]{microtype}
\usepackage{enumitem}
\usepackage{tabularx,float}
\usepackage[colorlinks=true,linkcolor=blue,citecolor=blue,urlcolor=blue]{hyperref}
\hypersetup{pdftitle={Sharp spectral norm concentration and regularization of sparse random tensors},
pdfauthor={Zhixin Zhou and Yizhe Zhu}}
\numberwithin{equation}{section}
\newtheorem{theorem}{Theorem}[section]
\newtheorem{lemma}[theorem]{Lemma}
\newtheorem{proposition}[theorem]{Proposition}
\newtheorem{corollary}[theorem]{Corollary}
\theoremstyle{definition}
\newtheorem{definition}[theorem]{Definition}
\theoremstyle{remark}

\newcommand{\E}{\mathbb E}
\newcommand{\Pp}{\mathbb P}
\newcommand{\R}{\mathbb R}
\newcommand{\one}{\mathbf 1}
\newcommand{\eps}{\varepsilon}
\title[Sparse tensor concentration and regularization]{Sharp spectral norm concentration and regularization of sparse random tensors}
\author{Zhixin Zhou}
\address{Alpha Benito Research}
\email{zhixin0825@gmail.com}

\author{Yizhe Zhu}
\address{Department of Mathematics, University of Southern California}
\email{yizhezhu@usc.edu}

\date{September 25, 2026}
\begin{document}
\begin{abstract}
We prove sharp spectral-norm bounds and construct a mean-fitting regularization procedure for sparse random tensors of fixed order $k$. For an $n\times\cdots\times n$ tensor with independent Bernoulli$(p)$ entries, the centered norm is $O(\sqrt{np})$ with polynomially high probability whenever $np\ge c\log n$. This removes the logarithmic loss in \cite{ZZ21}. We also determine the expected centered norm at every sparsity level: for every $0<p\le1/2$, it has order $\min\{n^kp,\sqrt{d_{\mathrm{eff}}}\}$, where $d_{\mathrm{eff}}$ is an effective degree accounting for unusually large fiber counts. The upper bounds extend to tensors with deterministic entrywise weights and unequal sampling probabilities. Analogous two-sided bounds hold for the Friedman--Wigderson model at every degree; its expected centered norm has order $\sqrt{d_{\mathrm{eff}}}$. We also give an all-sparsity regularization procedure for independent Bernoulli tensors with unknown, possibly unequal means in $[0,p]$. Our light-fitting regularization estimates the original mean with minimax expected spectral-norm error of order $\min\{\sqrt{np},pn^{k/2}\}$ for every known $0\le p\le1$. The procedure fits light multilinear residuals under an entrywise constraint on the candidate mean, rather than deleting high-degree fibers.
\end{abstract}
\maketitle

\section{Introduction}\label{sec:intro}

Many data sets record interactions among several objects rather than
pairs. Tensors provide a natural representation of such observations,
from hyperedges in a random hypergraph~\cite{ZZ21} to sampled entries
of a multiway array~\cite{JO14,YZ17}. The observation is random, while its mean describes
the underlying interaction pattern. A basic task is to distinguish
that pattern from sampling noise.

Controlling individual entries is not enough when the quantities of
interest combine many entries. For matrices, operator-norm bounds
are central to spectral methods for random graphs and
networks~\cite{FO05,LR15}. For tensors, the spectral norm is the
largest absolute value of a multilinear form on a product of unit
spheres. Applied to a centered tensor, it measures the largest
fluctuation over all such product directions. It controls edge-count
deviations in random
hypergraphs~\cite{FW95,ZZ21}, noise in spiked tensor
models~\cite{MR14}, and perturbations in latent-variable
learning~\cite{AGHKT14}. Related estimates arise in tensor
completion~\cite{HZ21,JO14,YZ17}, multilayer community
detection~\cite{LCL20}, and the study of spherical spin
glasses~\cite{ABC13,DM24}.

Sparsity makes this problem more subtle. A one-dimensional \emph{fiber}
is obtained by fixing all but one index. Even when the entries are
independent, a few unusually large fiber counts can dominate the
spectral norm. This leads to two questions: how large
are the fluctuations of a sparse random tensor, and, when the raw
observation fails to concentrate, can its mean still be estimated
at the natural scale?

We begin with tensors of fixed order $k$ and dimension
$n\times\cdots\times n$. \reviewblue{Write $P=\E T$ for the mean tensor.}
In the independent Bernoulli$(p)$ model, $d=np$ is the expected number
of nonzero entries in a fiber. The centered tensor norm dominates the Euclidean norm
of every centered fiber, so $\sqrt d$ is the natural
fluctuation scale when $p\le1/2$ and $d\gtrsim\log n$.
Our first result gives the matching upper bound: for fixed $c,r>0$,
\[
 d\ge c\log n
 \quad\Longrightarrow\quad
 \Pp\{\|T-\reviewblue{P}\|>C_{k,r,c}\sqrt d\}\le n^{-r}.
\]
For $k\ge3$, this removes the factor $(\log n)^{k-2}$ in
Zhou and Zhu~\cite{ZZ21}. The bound also holds for deterministic
entrywise weights and unequal Bernoulli probabilities
(Corollary~\ref{cor:weighted}). Both the scale $\sqrt d$ and the
logarithmic-degree condition are optimal in order.

To describe the full sparsity range, one must account for both local
irregularities and the possibility of an empty observation.
Below logarithmic degree, the largest fiber count can greatly
exceed its expectation. Theorem~\ref{thm:homogeneous-allp}
identifies the fluctuation scale in this regime. At $p=1/n$,
the centered norm has order $\sqrt{\log n/\log\log n}$ rather
than order one, extending familiar matrix behavior~\cite{S00,BBK19}
to higher-order tensors. At the other extreme, when the expected
total number of nonzero entries is small, the tensor is often empty
and rare nonempty observations determine the expected norm. Thus a
description of the expected norm must distinguish large-fiber
fluctuations from the effect of rare observations.

We also ask how these fluctuations change when the total number of
sampled edges is fixed in advance.
Theorem~\ref{thm:fw} gives the corresponding result for the
Friedman--Wigderson model, in which a fixed number of ordered edges
are sampled with replacement. The draws are independent, but their
occupancy counts are dependent. The same effective degree
$d_{\mathrm{eff}}$, defined in \eqref{eq:hom-effective-degree},
governs the centered norm. For every positive edge count, both its
expectation and its value with high probability have order
$\sqrt{d_{\mathrm{eff}}}$. No empty-observation correction is needed.

Failure of concentration need not imply failure of estimation.
It shows that the raw observation can be a poor estimate of its mean,
but does not rule out a more accurate function of the data.
The aim of regularization is to suppress exceptional fluctuations
without introducing a comparable bias. Our regularization result
estimates the original mean at the natural scale even when the
observation itself fails to concentrate.
Suppose the independent Bernoulli entries have an unknown mean
tensor $P$ with entries in $[0,p]$,
where $\reviewblue{0\le p\le1}$ is a known upper bound. Theorem~\ref{thm:minimax}
shows that the minimax expected error is
\[
 \inf_{\widehat P}\sup_{P\in[0,p]^{[n]^k}}
 \E_P\|\widehat P-P\|\asymp_k
 \min\{\sqrt{np},pn^{k/2}\}.
\]
\reviewblue{Our light-fitting estimator attains the upper bound;
the matching lower bound was established by
Zhou and Zhu~\cite[Theorem~2.7]{ZZ21}.
The fit has exponentially many net constraints, so the result is
a statistical guarantee, not a polynomial-time guarantee.}

A natural open question left by our analysis is whether, for fixed
$k\ge3$ and known $0\le p\le1$, a polynomial-time regularization
algorithm can attain the same minimax rate in expected spectral-norm
loss, uniformly over $P\in[0,p]^{[n]^k}$. Such an algorithm may
estimate the mean directly; it need not be a deletion or
degree-trimming rule.

\medskip
\noindent\textit{The proof mechanism.}
The upper bound must hold in every product direction, including
directions concentrated on a few coordinates.
For fixed unit vectors, the centered multilinear form has variance
at most $p$. A uniform bound over a net also requires control of
the largest coefficients.
Following the Kahn--Szemer\'edi method~\cite{FKS89} and
Zhou and Zhu~\cite[Section~4]{ZZ21}, we split index tuples into
\emph{light} and \emph{heavy} according to the size of their
coordinate products. Bernstein's inequality controls the light
part. The heavy part is the main difficulty.

We group coordinates of comparable magnitude into classes.
One class from each mode determines a box of tensor entries.
The earlier discrepancy bound~\cite[Lemma~4.4]{ZZ21} controls boxes
separately. When these bounds are summed, the cost of choosing a
shared coordinate class is charged every time that class appears.
In the difficult cases, this repeated charge produces a factor
$(\log n)^{k-2}$. Our joint discrepancy estimate controls a
collection of boxes at once, charging each shared class only once.
Together with fiber-degree bounds and the unit-length constraints,
it gives a single summation argument for every $k\ge2$
(Proposition~\ref{prop:heavy}).

The heavy estimate is deterministic once the degree and discrepancy
bounds hold. This allows the same argument to handle the dependent
occupancy counts in the Friedman--Wigderson model~\cite{FW95}.
A joint exponential-moment bound supplies the required discrepancy
estimate, yielding a second-eigenvalue bound without the
logarithmic loss in~\cite{FW95}.

\subsection{Notation and definitions}\label{sec:notation}

We follow the tensor notation of~\cite{ZZ21}. Write
$[n]=\{1,\ldots,n\}$ and denote an order-$k$ tensor by
$T=(t_{\mathbf i})_{\mathbf i\in[n]^k}$, where
$\mathbf i=(i_1,\ldots,i_k)$ is a multi-index.
A sum over $\mathbf i$ runs over $[n]^k$ unless stated otherwise.
For a vector $x_j\in\R^n$, its $i$th coordinate is $x_{j,i}$.
We write $\one=(1,\ldots,1)\in\R^n$.

For a tensor $A$, write $A=(a_{\mathbf i})$.
For the mean tensor $P=\E T$, write $P=(p_{\mathbf i})$,
so $p_{\mathbf i}=\E t_{\mathbf i}$.

\begin{definition}[Frobenius inner product and spectral norm]
For tensors $T,A\in\R^{n^k}$, define
\[
 \langle T,A\rangle
 =\sum_{\mathbf i}t_{\mathbf i}a_{\mathbf i},
 \qquad
 \|T\|_F=\langle T,T\rangle^{1/2}.
\]

For vectors $x_1,\ldots,x_k\in\R^n$, their outer product has coordinates
\[
 (x_1\otimes\cdots\otimes x_k)_{\mathbf i}
 =\prod_{j=1}^k x_{j,i_j},
 \qquad \mathbf i=(i_1,\ldots,i_k)\in[n]^k.
\]

The associated multilinear form and the spectral norm are
\begin{align}
 T(x_1,\ldots,x_k)&=\langle T,x_1\otimes\cdots\otimes x_k\rangle,
 \label{eq:form}\\
 \|T\|&=\sup_{\|x_1\|_2=\cdots=\|x_k\|_2=1}
                 |T(x_1,\ldots,x_k)|.
 \label{eq:norm}
\end{align}
\end{definition}

The norm in \eqref{eq:norm} is also called the \textit{injective norm}.
For $k=2$ it is the usual matrix operator norm. For higher orders,
it is not the matrix operator norm obtained by unfolding the tensor.
Since
\[
 \|x_1\otimes\cdots\otimes x_k\|_F^2
 =\prod_{j=1}^k\|x_j\|_2^2,
\]
the Cauchy--Schwarz inequality gives $\|T\|\le\|T\|_F$.
More generally, multilinearity implies
\begin{equation}\label{eq:homogeneity}
 |T(x_1,\ldots,x_k)|\le\|T\|\prod_{j=1}^k\|x_j\|_2.
\end{equation}

We write $\log_+u=\max\{0,\log u\}$. The letter $C$ denotes a positive
constant that may change from line to line; subscripts indicate its
allowed dependence. The order $k$ is fixed throughout.

\subsection{Related work}\label{sec:related}

\smallskip\noindent\emph{Spectral norms of random graphs and tensors.}
The light--heavy method originates in the Kahn--Szemer\'edi approach
to random regular graphs~\cite{FKS89}. Feige and Ofek~\cite{FO05}
applied spectral and discrepancy methods to sparse random graphs,
including the logarithmic-degree regime. Subsequent work gives spectral
estimates for a wide range of random graph and hypergraph
models~\cite{BFSU98,LR15,DJPP13,CGJ18,TY19,Z23,DWZ25}.

The constant-degree matrix case is classical.
Krivelevich and Sudakov~\cite[Corollary~1.2]{KS03} showed that the
adjacency matrix of $G(n,d/n)$, for fixed $d>0$, has norm
$(1+o(1))\sqrt{\log n/\log\log n}$ with high probability.
Benaych-Georges, Bordenave, and Knowles~\cite{BBK19} obtained
sharp extreme-eigenvalue estimates for centered inhomogeneous
adjacency matrices below logarithmic degree. Applying their
Corollary~1.4 to the bipartite Hermitian dilation of an independent
Bernoulli matrix gives the same leading scale for
$\|T-\reviewblue{P}\|$ when $p\asymp1/n$.
For the expected norm, Seginer~\cite[Corollary~2.2]{S00} compares
the expected operator norm of a matrix with centered i.i.d. entries
to the expected largest Euclidean row and column norms, up to
universal constants.
Thus the constant-degree case of Theorem~\ref{thm:homogeneous-allp}
recovers classical matrix behavior when $k=2$ and extends it to
higher-order tensors.

For random tensors, Nguyen, Drineas, and Tran~\cite{NDT15} studied
spectral norms through entropy and concentration estimates for
sparse and spread vectors, with applications to tensor sparsification.
Tomioka and Suzuki~\cite{TS14} gave a covering-number argument under
sub-Gaussian assumptions on fixed multilinear evaluations.
For Bernoulli$(p)$ entries, the sub-Gaussian bound is not uniform at
the variance scale $p(1-p)$ as $p$ tends to zero.
For real and complex tensors with independent standard Gaussian
entries, Dartois and McKenna~\cite{DM24} obtained
high-probability upper bounds with refined leading constants by
counting critical points of the associated Gaussian multilinear form.
A complementary moment method bounds expected injective norms in
Gaussian and certain non-Gaussian models~\cite{DM26}.

For independent entries, Boedihardjo~\cite{Boe24} obtained
expected injective-norm bounds in terms of the entry variances,
first for Gaussian tensors and then for bounded entries, together
with concentration about the expectation.
In the homogeneous Bernoulli model, his Corollary~1.5 gives
\[
 \E\|T-\reviewblue{P}\|
 \le C_k\bigl(\sqrt{np(1-p)}+(\log n)^2\bigr).
\]
The additive logarithmic term does not yield the $\sqrt{np}$ scale
when $np\asymp\log n$.

Lucca and Pesenti~\cite[Theorem~1.3]{LP26} determined the
spectral gap of sparse Erd\H{o}s--R\'enyi hypergraphs.
For fixed $k\ge3$, let $T$ be the symmetric $0$--$1$ adjacency
tensor of a $k$-uniform hypergraph on $n$ vertices, with each
$k$-element subset included independently with probability $p$.
If $pn^{k/2}\to\infty$ as $n\to\infty$, then with high probability,
\[
 \|T\|-\|T-\reviewblue{P}\|=(1-o(1))pn^{k/2}.
\]
This gives a spectral-gap asymptotic in a regime extending below
$p=1/n$, whereas our main result gives a sharp $O(\sqrt{np})$
centered-norm bound for independent-entry tensors when $np\ge c\log n$.

\smallskip\noindent\emph{Sampling and structured tensors.}
Concentration for sampled tensors also arises in tensor completion
and network estimation. Jain and Oh~\cite{JO14}
proved a spectral-norm bound for randomly sampled symmetric
third-order tensors in their analysis of tensor completion.
Lei, Chen, and Lynch~\cite{LCL20} obtained a spectral-norm
concentration bound for the centered third-order adjacency tensor
in a multilayer network model.
For tensor sparsification, Xia and Yuan~\cite{XY21}
established a spectral-norm concentration inequality for tensors
obtained by entrywise Bernoulli sampling and reweighting.

Guan, Jiang, and Li~\cite{GJL25} observed that a weighted Bernoulli
spectral-norm bound without logarithmic loss would remove the
remaining order-dependent logarithmic factor from their tensor
robust-PCA recovery conditions.
Corollary~\ref{cor:weighted} proves Conjecture~7.1 in \cite{GJL25}.

The choice of tensor norm also matters.
Yuan and Zhang~\cite{YZ17} proved concentration
inequalities for sampled tensors under an \emph{incoherent spectral
norm}. It restricts the coordinates of the test vectors in all but
two modes, allowing concentration below $np\asymp 1$.
A sharper bound under this norm, with an application to tensor
completion, was obtained in~\cite{LMSZ26}. These bounds restrict
the test vectors, whereas our spectral norm ranges over all
Euclidean unit vectors.

Dependent-entry tensors require different methods.
For tensor products of independent random vectors, results include
concentration for convex Lipschitz functions under bounded-coordinate
assumptions~\cite{V20}, as well as Hanson--Wright-type inequalities
and concentration of Euclidean norms of linear images when the
coordinates are independent and sub-Gaussian~\cite{BKW22,V20}.
Nonasymptotic injective-norm bounds have also been established for Gaussian
and Rademacher tensor series with correlated entries~\cite{BGJLR24}.
For centered empirical moment tensors and their asymmetric
counterparts, empirical-process methods give sharp injective-norm
bounds under sub-Gaussian assumptions~\cite{ACS25,CS26}.
Here the observations are independent, but the tensor entries
share random factors and are generally dependent. These results
complement the independent-entry estimates considered here.

\smallskip\noindent\emph{Regularization and mean estimation.}
Regularization addresses the failure of concentration caused by
unusually large degrees. Feige and Ofek~\cite{FO05} showed that
removing high-degree vertices restores concentration in homogeneous
random graphs. Le, Levina, and Vershynin~\cite{LLV17} extended this
principle to inhomogeneous graphs: reducing the weights of edges
incident to high-degree vertices restores the $O(\sqrt d)$ scale
around the original mean, where $d=n\max_{i,j}p_{ij}$.
Wang~\cite{W26} obtained refined spectral bounds for degree-regularized
restrictions of centered inhomogeneous Bernoulli matrices, with
explicit dependence on variance profiles and degree cutoffs, using
nonbacktracking operators.

A complementary line of work concerns centered matrices with i.i.d.
finite-variance entries. Rebrova and Vershynin~\cite{RV18} showed that
zeroing a small submatrix brings the norm to its natural scale.
Jain, Sah, and Sawhney~\cite{JSS22} obtained the optimal dependence
on the size of the removed submatrix and a deterministic
polynomial-time algorithm. These results regularize centered noise,
whereas our estimation problem allows an unknown inhomogeneous mean.
Related estimation results include the minimax rates for low-rank
connectivity-matrix and subspace estimation in random bipartite graphs
obtained by Zhou et al.~\cite{zhou2021rate}.

For tensors, Zhou and Zhu~\cite[Section~2.3 and Theorem~2.11]{ZZ21}
regularized matrix unfoldings by degree trimming. Our result instead
estimates the original unknown mean directly in the tensor spectral
norm, attaining the minimax error in Theorem~\ref{thm:minimax} at every
sparsity level.

\subsection*{Organization of the paper}
Section~\ref{sec:main} presents the main results in four steps:
logarithmic-degree concentration, homogeneous all-sparsity
fluctuations, fixed-edge bounds, and mean-fitting regularization.
The proofs follow the same order: Bernoulli
concentration in Section~\ref{sec:proof},
homogeneous all-sparsity bounds in Section~\ref{sec:homogeneous-allp},
fixed-edge all-degree bounds in Section~\ref{sec:fw-proof}, and
the regularization guarantee in
Section~\ref{sec:minimax}.

\section{Main results}\label{sec:main}

\subsection{Concentration at logarithmic degree}

Let $T$ have independent entries
\begin{equation}\label{eq:model}
 t_{\mathbf i}\sim\operatorname{Bernoulli}(p),
 \qquad \mathbf i\in[n]^k.
\end{equation}
\reviewblue{The mean tensor is}
\[
 P=\E T=p\,\one^{\otimes k}.
\]
Independence in \eqref{eq:model} is over all ordered index tuples.
Set $d=np$; this is the expected number of nonzero entries in any
one-dimensional fiber.

\begin{theorem}\label{thm:main}
Let $k\ge2$ be a fixed integer. For any $c>0$ and $r>0$, there exists a
constant $C_{k,r,c}$ depending only on $k,r,c$ such that, for all
$n\ge2$ and $p\in(0,1]$, if $ d=np\ge c\log n$,  then the tensor in \eqref{eq:model} satisfies
\begin{equation}\label{eq:main}
 \Pp\{\|T-\reviewblue{P}\|>C_{k,r,c}\sqrt d\}\le n^{-r}.
\end{equation}
\end{theorem}

The constant $c>0$ may be arbitrarily small, but it is fixed.
The multiplicative constant $C_{k,r,c}$ may depend on $k$, $r$, and $c$.

\medskip
\noindent\textit{Optimality of the scale and sparsity condition.}
The tensor norm always dominates the largest $\ell_2$-norm of its
one-dimensional fibers. In particular,
\[
 \|T-\reviewblue{P}\|\ge\max_f\|f\|_2,
\]
where the maximum ranges over all one-dimensional fibers $f$ of
$T-\reviewblue{P}$. This follows by fixing all but one argument at standard
basis vectors. For $p\le1/2$ and $d\ge c\log n$, even a fixed fiber
has norm at least $\sqrt{d/2}$ with probability at least
$1-e^{-d/8}$ by the Chernoff bound, showing that the $\sqrt d$ scale
is optimal. If $1\le d=np=o(\log n)$, the maximum fiber norm is
$\omega(\sqrt d)$ with high probability, so
$\|T-\reviewblue{P}\|/\sqrt d\to\infty$ in probability as $n\to\infty$.
Thus the condition $d\ge c\log n$ is optimal in order.

\reviewblue{The proof of Theorem~\ref{thm:main} is in
Section~\ref{sec:proof}.}

\subsection{Homogeneous fluctuations at every sparsity level}

We next remove the logarithmic-degree assumption in the homogeneous
model. \reviewblue{As before, write $d=np$.}

For $p>0$, define the \emph{effective degree}
\begin{equation}\label{eq:hom-effective-degree}
 d_{\mathrm{eff}}
 =\max\left\{1,\ d+
       \frac{\log n}{\log(e+\log n/d)}\right\}.
\end{equation}
The effective degree describes the scale of large fiber
counts; it does not modify $T$. Its logarithmic term accounts for
unusually large counts, while the lower cutoff at $1$ accounts for
isolated nonzero entries.

\begin{theorem}[\reviewblue{Matching homogeneous bounds at every sparsity level}]
\label{thm:homogeneous-allp}
Fix $k\ge2$. Let $n\ge2$ and let $T$ have independent
$\operatorname{Bernoulli}(p)$ entries, where $0<p\le1/2$.
\reviewblue{Write $P=p\,\one^{\otimes k}$ for its mean tensor.}
There are constants $c_k,C_k>0$ such that
\begin{equation}\label{eq:hom-expectation}
 c_k\min\{n^kp,\sqrt{d_{\mathrm{eff}}}\}
 \le \E\|T-\reviewblue{P}\|
 \le C_k\min\{n^kp,\sqrt{d_{\mathrm{eff}}}\}.
\end{equation}
For every $r>0$, there is a constant $C_{k,r}$ such that
\begin{equation}\label{eq:hom-upper-tail}
 \Pp\{\|T-\reviewblue{P}\|>
          C_{k,r}\sqrt{d_{\mathrm{eff}}}\}\le n^{-r}.
\end{equation}
There is also the lower bound
\begin{equation}\label{eq:hom-lower-tail}
 \Pp\{\|T-\reviewblue{P}\|<
          c_k\sqrt{d_{\mathrm{eff}}}\}
 \le \exp(-n^kp/8)+\exp(-n^{5/8}).
\end{equation}
\end{theorem}

\reviewblue{For $0<p\le1/2$,} three regimes clarify the bound. When $d\gtrsim\log n$,
the effective degree has order $d$, recovering the scale $\sqrt d$.
If $p=n^{-\alpha}$ for fixed $1<\alpha<k$, the centered norm is of
constant order with probability tending to one. When $n^kp\to0$,
the tensor is empty with probability tending to one. On that event,
the centered norm is $pn^{k/2}$, but its expectation has order
$n^kp$, because rare nonempty observations dominate the expectation.

\smallskip\noindent\emph{Comparison with \cite{LP26}.}
Lucca and Pesenti~\cite[Theorem~1.4]{LP26} bound the expected
injective norm of a centered tensor with general independent entries
by its expected largest absolute entry, provided each entry is
nonzero with probability at most $n^{-1-\eps}$ for some fixed $\eps\in(0,1)$.
For homogeneous Bernoulli entries, their bound becomes
\[
 \E\|T-\reviewblue{P}\|\le C_{k,\eps}\min\{1,n^kp\},
 \qquad p\le n^{-1-\eps}.
\]
Theorem~\ref{thm:homogeneous-allp} matches this scale in that regime
and identifies the transition as $p$ approaches $1/n$. In particular,
at $p=1/n$ it gives
\[
 \E\|T-\reviewblue{P}\|\asymp_k\sqrt{\log n/\log\log n}
 \qquad(n\to\infty).
\]
For $k\ge3$, this settles the centered Bernoulli case singled out
after \cite[Conjecture~1.8]{LP26}, relating the expected injective
norm to the expected largest $\ell_2$ norm of a fiber of $T-\reviewblue{P}$.
Our result treats homogeneous Bernoulli tensors at every sparsity
level; it does not resolve their conjecture for general i.i.d. entries.
The proof is in Section~\ref{sec:homogeneous-allp}.

The upper bounds also hold for rectangular tensors with deterministic
entrywise weights and unequal sampling probabilities.

\begin{corollary}[Weighted Bernoulli sampling at every sparsity level]
\label{cor:weighted}
Let $k\ge2$ be fixed, let $n\ge2$ and $1\le n_1,\ldots,n_k\le n$
be integers, and let $0<p\le1$. Let
$A=(a_{\mathbf i})\in\R^{n_1\times\cdots\times n_k}$ be deterministic.
Let $\Omega=(\omega_{\mathbf i})$ have independent
$\operatorname{Bernoulli}(p_{\mathbf i})$ entries with
$0\le p_{\mathbf i}\le p$.
Set $d=np$ and define $d_{\mathrm{eff}}$ by
\eqref{eq:hom-effective-degree}.
For every $r>0$, there is a constant $C_{k,r}$ such that
\begin{equation}\label{eq:weighted}
 \Pp\left\{\|A\odot(\Omega-\E\Omega)\|>
 C_{k,r}\|A\|_\infty\sqrt{d_{\mathrm{eff}}}\right\}\le n^{-r}.
\end{equation}
Moreover, there is a constant $C_k$ such that
\begin{equation}\label{eq:weighted-expectation}
 \E\|A\odot(\Omega-\E\Omega)\|
 \le C_k\|A\|_\infty
       \min\left\{\sum_{\mathbf i}p_{\mathbf i},
                      \sqrt{d_{\mathrm{eff}}}\right\}.
\end{equation}
Here $\odot$ denotes entrywise multiplication,
$\|A\|_\infty=\max_{\mathbf i}|a_{\mathbf i}|$, and the sum is over
$[n_1]\times\cdots\times[n_k]$.
\end{corollary}

Here $n$ and $p$ are upper bounds, so every fiber of $\Omega$ has
expected degree at most $d=np$; neither bound needs to be attained.
No lower bound on $np$ is required. When $np\ge c\log n$,
we have $d_{\mathrm{eff}}\le C_cnp$, so \eqref{eq:weighted} recovers the
$\|A\|_\infty\sqrt{np}$ bound. At the sparsest end,
\eqref{eq:weighted-expectation} also accounts for the small expected
number of sampled entries. These are upper bounds; arbitrary weights
and sampling probabilities need not admit a matching lower bound.
The proof is in Section~\ref{sec:homogeneous-allp}.

\subsection{The Friedman--Wigderson fixed-edge model}\label{sec:fw}

The same fluctuation scale governs the random hypergraph model of
Friedman and Wigderson~\cite[Section~6]{FW95}. Here the number of
sampled edges is fixed: the draws are independent, but the resulting
occupancy counts are dependent.

Let $m\ge1$ be an integer, and let $E_1,\ldots,E_m$ be independent
uniform random elements of $[n]^k$. Repetitions of an edge and
repeated coordinates within an edge are allowed. Define the
occupancy tensor $X=(X_{\mathbf i})$ by
\begin{equation}\label{eq:fw-model}
 X_{\mathbf i}=\sum_{\nu=1}^m\one\{E_\nu=\mathbf i\},
 \qquad d=\frac{m}{n^{k-1}},\qquad
 P=\E X=\frac dn\one^{\otimes k}.
\end{equation}

Thus $\E X_{\mathbf i}=d/n$. The centered norm $\|X-P\|$ is the
Friedman--Wigderson second eigenvalue. Here $d$ specifies the
centering, not exact fiber regularity.

The total edge count is fixed; individual fiber counts are random
and have expectation $d$.

\begin{theorem}[Matching fixed-edge bounds at every degree]\label{thm:fw}
Fix $k\ge2$. Let $n\ge2$ and $m\ge1$ be integers, and let $X$ be
the occupancy tensor in \eqref{eq:fw-model}. Define
$d_{\mathrm{eff}}$ by \eqref{eq:hom-effective-degree}, now with
$d=m/n^{k-1}$. There are constants $c_k,C_k>0$ such that
\begin{equation}\label{eq:fw-expectation}
 c_k\sqrt{d_{\mathrm{eff}}}
 \le \E\reviewblue{\|X-P\|}
 \le C_k\sqrt{d_{\mathrm{eff}}}.
\end{equation}
For every $r>0$, there is a constant $C_{k,r}$ such that
\begin{equation}\label{eq:fw-upper-tail}
 \Pp\{\reviewblue{\|X-P\|}>C_{k,r}\sqrt{d_{\mathrm{eff}}}\}
 \le n^{-r}.
\end{equation}
Moreover, there are absolute constants $c_0,C_0>0$ such that
\begin{equation}\label{eq:fw-lower-tail}
 \Pp\{\reviewblue{\|X-P\|}<c_0\sqrt{d_{\mathrm{eff}}}\}
 \le C_0\exp(-c_0\sqrt n).
\end{equation}
If $m=0$, then $X$ and its centered norm are identically zero.
\end{theorem}

For fixed $k\ge3$ in the same model,
Friedman and Wigderson~\cite[Theorem~1.5]{FW95} proved, for
$d\ge Ck\log n$, that with high probability
 $\reviewblue{\|X-P\|}\le (C\log n)^{k/2}\sqrt d$.
Since $d_{\mathrm{eff}}\asymp d$ when $d\gtrsim\log n$,
Theorem~\ref{thm:fw} removes the factor $(\log n)^{k/2}$ in that
range and identifies the larger fluctuation scale below it.
Unlike the Bernoulli tensor, $X$ is never empty when $m\ge1$,
so its expectation requires no empty-tensor correction. The proof
is in Section~\ref{sec:fw-proof}.

\subsection{Regularization by light fitting}

Let $k\ge2$, $n\ge16$, and let $0\le p\le1$ be a known
upper bound on the entries of an unknown mean tensor
$P=(p_{\mathbf i})=\E T$.
We observe independent entries
$t_{\mathbf i}\sim\operatorname{Bernoulli}(p_{\mathbf i})$,
where $0\le p_{\mathbf i}\le p$, and estimate $P$ under the spectral norm.
Write $\mathcal P(p)=[0,p]^{[n]^k}$.
An estimator may be any measurable function of $T$ and the known parameter $p$ of $\mathcal P$.
Its minimax expected error is

\begin{equation}\label{eq:mm-risk}
 \inf_{\widehat P}\sup_{P\in\mathcal P(p)}
                  \E_P\|\widehat P(T)-P\|.
\end{equation}

\begin{theorem}[\reviewblue{The minimax rate for every $0\le p\le1$}]
\label{thm:minimax}
There are constants $0<c_k\le C_k<\infty$ such that

\begin{equation}\label{eq:mm-rate}
\begin{aligned}
 c_k\min\{\sqrt{np},\,p n^{k/2}\}
 &\le \inf_{\widehat P}\sup_{P\in\mathcal P(p)}
                   \E_P\|\widehat P(T)-P\|\le C_k\min\{\sqrt{np},\,p n^{k/2}\}.
\end{aligned}
\end{equation}

\reviewblue{The estimator $\widehat P$ in
Algorithm~\ref{alg:light-fitting} attains the upper bound up to its constant.}
\end{theorem}

\smallskip\noindent\emph{Why direct fiber trimming fails.}
Degree regularization for graphs suggests trimming fibers with
unusually large counts. For tensors, however, the resulting bias
relative to the original mean can overwhelm the desired error scale.
To see this, take $k\ge3$ and homogeneous
entries with $p=d/n$, where $d>0$ is fixed.
Fix an integer cutoff $M=\lceil Cd\rceil$, with $C>1$ fixed, and
delete every first-mode fiber whose count exceeds $M$.
Write $\widetilde T$ for the resulting tensor and $P=\E T$.
There are $n^{k-1}$ disjoint such fibers, each with count
$S\sim\operatorname{Bin}(n,d/n)$. Their removed entries contribute
with the same sign. Testing against $\one$ in each mode gives

\begin{align*}
\E\|\widetilde T-P\|&\ge\reviewblue{n^{-k/2}     \left|\E(\widetilde T-P)(\one,\ldots,\one)\right|}=n^{k/2-1}\E\bigl[S\one_{\{S>M\}}\bigr]
 \asymp_{d,M} n^{k/2-1}.
\end{align*}

Here
$\E[S\one_{\{S>M\}}]
=d\Pp\{\operatorname{Bin}(n-1,d/n)\ge M\}$
tends to $d\Pp\{\operatorname{Pois}(d)\ge M\}>0$.
The loss per fiber therefore remains positive as $n$ grows.
Across all fibers it produces a bias of order $n^{k/2-1}$,
which diverges for $k\ge3$ even though $\sqrt{np}=\sqrt d$
stays constant. For matrices, the factor $n^{k/2-1}$ is absent.
Deleting the initially over-threshold fibers in all modes only
removes more mass. This example concerns direct fixed-cutoff
trimming; bias-corrected fiber procedures require a separate analysis.

To estimate the original mean at the desired scale, we search over
candidate means $A\in\mathcal P(p)$ and compare their residuals
$T-A$ along product directions. For the small coordinate products,
the noise concentrates, so we fit these light contributions.
The true mean is an admissible candidate, allowing the fitted
light residual to be compared with the residual at $P$.
For the remaining heavy tuples, both $A$ and $P$ have entries
in $[0,p]$, so each entry of $A-P$ has magnitude at most $p$.
Combining these controls gives the $\sqrt{np}$ scale in expectation.
The zero estimator covers the regime where $pn^{k/2}$ is smaller.

\smallskip\noindent\emph{The light-fitting estimator.}
Fix in advance a deterministic $1/(2k)$-net $\mathcal T$ of the
unit sphere in $\R^n$ with $|\mathcal T|\le(1+4k)^n$;
Lemma~\ref{lem:net} ensures its existence.

For $p>0$, set $\tau=\sqrt{p/n}$ and, for
$x=(x_1,\ldots,x_k)\in\mathcal T^k$, define
\[
 \mathcal L(x)=\{\mathbf i\in[n]^k:
                   \bigl|\prod_{j=1}^k x_{j,i_j}\bigr|\le\tau\}.
\]

Set $\widehat P=0$ when $pn^{k/2}\le\sqrt{np}$, including $p=0$.
Otherwise, minimize the largest absolute light residual over the
product net:

\begin{equation}\label{eq:mm-fit}
\begin{aligned}
 F_T(A)&=\max_{x\in\mathcal T^k}
 \left|\sum_{\mathbf i\in\mathcal L(x)}
          (t_{\mathbf i}-\reviewblue{a_{\mathbf i}})
          \reviewblue{\prod_{j=1}^k x_{j,i_j}}\right|,\qquad
 \reviewblue{\widehat P}\in
              \operatorname*{arg\,min}_{A\in\mathcal P(p)}F_T(A).
\end{aligned}
\end{equation}
\reviewblue{Write $\widehat P=(\widehat p_{\mathbf i})$.}
Continuity and compactness guarantee a minimizer.
Fix a deterministic tie-breaking rule; since the sample space is
finite, the resulting estimator is measurable. The light sets
depend only on $p$ and the net vectors, not on $A$, so the fit is
a linear program. \reviewblue{Algorithm~\ref{alg:light-fitting}
implements the two branches defining $\widehat P$.}

\begin{table}[H]
\renewcommand{\tablename}{Algorithm}
\caption{Light-fitting regularization of a Bernoulli tensor}
\label{alg:light-fitting}
\small
\renewcommand{\arraystretch}{1.15}
\begin{tabularx}{\linewidth}{@{}p{0.12\linewidth}X@{}}
\hline
Input & An observed order-$k$ tensor $T$, $n\ge16$, $k\ge2$,
and a known bound $\reviewblue{0\le p\le1}$ on its Bernoulli means.\\
Setup & Use the fixed net $\mathcal T$ above. \reviewblue{For $p>0$,
define $\tau$ and $\mathcal L(x)$ as above.}\\
Step 1 & If $pn^{k/2}\le\sqrt{np}$, return $\widehat P=0$
and stop. This includes $p=0$.\\
Step 2 & Otherwise, solve the linear program
\[
\begin{aligned}
 \min_{A,z}\quad &z\\
 \text{subject to}\quad
 &0\le \reviewblue{a_{\mathbf i}}\le p\quad(\mathbf i\in[n]^k),\qquad z\ge0,\\
 &-z\le\sum_{\mathbf i\in\mathcal L(x)}
       (t_{\mathbf i}-\reviewblue{a_{\mathbf i}})
       \reviewblue{\prod_{j=1}^k x_{j,i_j}}\le z
       \quad(x\in\mathcal T^k).
\end{aligned}
\]
Choose $A$ by the fixed tie-breaking rule.\\
Output & \reviewblue{Return $\widehat P=A$.}\\
\hline
\end{tabularx}
\end{table}

The fit may assign positive mass
to observed zeros, so its output is a mean estimate rather than a
trimmed observation. Its exponentially many net constraints
preclude a polynomial-time claim here. Section~\ref{sec:minimax}
proves the upper bound; the lower bound follows from
Zhou and Zhu~\cite[Theorem~2.7]{ZZ21}.

\section{Proofs for the Bernoulli model}\label{sec:proof}

We separate light and heavy tuples. A product net and Bernstein's
inequality control the light part. A dyadic decomposition reduces the
heavy part to a deterministic estimate, whose hypotheses are bounds
on fiber degrees and joint box discrepancies.
Throughout this section, $T$ is the Bernoulli tensor in
\eqref{eq:model}, $P=\E T$, and $d=np$, except in the explicitly
deterministic statement and proof of Proposition~\ref{prop:heavy}.

\subsection{Discretization}\label{sec:discretization}

An $\eps$-net of the unit sphere in $\R^n$ approximates each point
within Euclidean distance $\eps$. We use a net on the sphere rather
than the lattice discretization in \cite[Section~4.1]{ZZ21}, reducing
the norm bound to finitely many scalar estimates.
\begin{lemma}[Euclidean sphere net {\cite[Corollary~4.2.11]{V26}}]\label{lem:net}
For $0<\eps<1$, the Euclidean unit sphere in $\R^n$ has a deterministic
$\eps$-net with at most $(1+2/\eps)^n$ elements.
\end{lemma}

Use the fixed net $\mathcal T$ from
Algorithm~\ref{alg:light-fitting}, with $\eps=1/(2k)$.
Its product $\mathcal T^k$ has at most $(1+4k)^{kn}$ elements.
For any tensor $\reviewblue{A}$, this net gives
\begin{equation}\label{eq:net-bound}
 \|\reviewblue{A}\|\le2\max_{(y_1,\ldots,y_k)\in\mathcal T^k}
                       |\reviewblue{A}(y_1,\ldots,y_k)|.
\end{equation}
Indeed, for unit vectors $x_j$, choose $y_j\in\mathcal T$ with
$\|x_j-y_j\|_2\le1/(2k)$. Replacing the vectors one at a time and
using \eqref{eq:homogeneity} gives
\[
 |\reviewblue{A}(x_1,\ldots,x_k)-\reviewblue{A}(y_1,\ldots,y_k)|
 \le\sum_{j=1}^k\|\reviewblue{A}\|\|x_j-y_j\|_2
 \le\frac12\|\reviewblue{A}\|.
\]
Taking the supremum proves \eqref{eq:net-bound}.

Put $\tau=\sqrt d/n$. For $x=(x_1,\ldots,x_k)$, define the set of light tuples by
\begin{equation}\label{eq:light-set}
 \mathcal L(x)
 =\left\{\mathbf i\in[n]^k:
             \left|\prod_{j=1}^k x_{j,i_j}\right|\le\tau\right\}.
\end{equation}
The remaining tuples are heavy. Since the threshold depends on the
coordinate product, we decompose each vector into coordinate classes
of comparable magnitude.

\subsection{Light tuples}\label{sec:light}

\begin{lemma}[The light part on a product net]\label{lem:light}
Fix $k$ and $R>0$. Put $W=T-P$ and $\tau=\sqrt d/n$.
Let $\mathcal T$ be a deterministic $1/(2k)$-net of the Euclidean unit sphere,
with $|\mathcal T|\le(1+4k)^n$. For a tuple
$x=(x_1,\ldots,x_k)\in\mathcal T^k$, put
\[
 L_W(x)=\sum_{\mathbf i\in\mathcal L(x)}
           \reviewblue{(t_{\mathbf i}-p)}\prod_jx_{j,i_j}.
\]
There is a constant $B_{k,R}$ such that
\[
 \Pp\left\{\max_{x\in\mathcal T^k}|L_W(x)|>
                         B_{k,R}\sqrt d\right\}\le n^{-R}.
\]
\end{lemma}
\begin{proof}
Fix a tuple in the product net. Its summands are independent,
centered, and bounded in absolute value by $\tau$. Their total variance
is at most
\[
 p\sum_{\mathbf i}\prod_jx_{j,i_j}^2
       =p\prod_j\|x_j\|_2^2=p.
\]
Bernstein's inequality~\cite[Theorem~2.9.5]{V26} gives
\[
 \Pp\{|L_W(x)|>B\sqrt d\}
 \le2\exp\left\{-\frac{B^2d}{2p+(2/3)\tau B\sqrt d}\right\}
 =2\exp\left\{-n\frac{B^2}{2+(2/3)B}\right\}.
\]
Choose $B=B_{k,R}$ so that
\[
 \frac{B^2}{2+(2/3)B}\ge k\log(1+4k)+R+1.
\]
A union bound over at most $(1+4k)^{kn}$ tuples then gives a failure
probability at most $2e^{-(R+1)n}\le n^{-R}$ for $n\ge2$.
\end{proof}

\subsection{Heavy tuples}\label{sec:heavy}

A \emph{box} is a set of the form
\[
 Q=S_1\times\cdots\times S_k,
 \qquad S_j\subseteq[n]\quad(j=1,\ldots,k).
\]
The set $S_j$ is the \emph{coordinate set} of $Q$ in mode $j$.

We first bound the nonnegative heavy sum for $T$; subtracting the
mean costs at most $\sqrt d$ by \eqref{eq:heavy-mean}.
The two probabilistic inputs are a bound on every fiber degree and
a joint discrepancy estimate for collections of boxes. On the event
where both hold, Lemma~\ref{lem:mass} lets us sum over the dyadic
levels without a loss proportional to their number.

\subsubsection{The bounded degree property}\label{sec:degrees}

A one-dimensional fiber is obtained by fixing all indices except one.
For a fiber $F\subseteq[n]^k$, its degree is
\[
 \deg(F)=\sum_{\mathbf i\in F}t_{\mathbf i}.
\]
These are degrees of $(k-1)$-tuples, not vertex degrees in a
$k$-uniform hypergraph. There are at most $kn^{k-1}$ fibers, each
with mean degree $d$.
\begin{lemma}[Maximum fiber degree]\label{lem:fibers}
For fixed $k\ge2$, $R>0$, and $c>0$, there is a constant
$\kappa=\kappa(k,R,c)>1$ such that, if $d=np\ge c\log n$, all
one-dimensional fibers of $T$ have degree at most $\kappa d$,
except on an event of probability at most
$n^{-R}$.
\end{lemma}
\begin{proof}
Each degree is $\operatorname{Bin}(n,p)$. Chernoff's bound gives,
for $\kappa>1$,
\[
 \Pp\{\operatorname{Bin}(n,p)\ge\kappa d\}
 \le\exp\{-d(\kappa\log\kappa-\kappa+1)\}.
\]
Choose $\kappa$ such that
$\kappa\log\kappa-\kappa+1\ge (R+k+1+\log_2 k)/c$. The union bound and
$d\ge c\log n$ then give a failure probability at most
\[
 kn^{k-1}n^{-(R+k+1+\log_2 k)}
 =n^{-R-2}\,k n^{-\log_2 k}\le n^{-R}.
\]
\end{proof}

\subsubsection{Simultaneous multibox discrepancy}\label{sec:entropy}

The upper-tail rate quantifies the rarity of a large box count:
for $u\gg\mu$, it has order $u\log(u/\mu)$.
Independence lets us add the rates for disjoint boxes.

For $\mu>0$ and $u\ge0$, define the upper-tail Poisson rate
\begin{equation}\label{eq:rate}
 I_\mu(u)=
 \begin{cases}
 u\log(u/\mu)-u+\mu,&u\ge\mu,\\
 0,&0\le u<\mu.
 \end{cases}
\end{equation}
For a binomial count $M$ with mean $\mu$, the classical Chernoff
bound~\cite[Theorem~2.3.1]{V26} gives
\[
 \Pp\{M\ge u\}\le e^{-I_\mu(u)},\qquad u\ge\mu.
\]
Below the mean, $I_\mu(u)=0$, so these counts contribute no upper-tail
discrepancy, and $e^{-I_\mu(u)}=1$ gives the trivial bound
$\Pp\{M\ge u\}\le1$.

\begin{lemma}[Exponential moment of the rate]\label{lem:rate}
If $M\sim\operatorname{Bin}(V,p)$, where $V\ge1$ and $\mu=pV>0$, then
\[
 \Pp\{I_\mu(M)>t\}\le e^{-t}\quad(t\ge0),
 \qquad
 \E\exp\bigl(I_\mu(M)/2\bigr)\le2.
\]
\end{lemma}
\begin{proof}
The rate is continuous and strictly increasing from zero to infinity
on $[\mu,\infty)$. For $t\ge0$, let $u_t\ge\mu$ solve
$I_\mu(u_t)=t$. Since $\{I_\mu(M)>t\}=\{M>u_t\}$, Chernoff's bound gives
\[
 \Pp\{I_\mu(M)>t\}
 \le\Pp\{M\ge u_t\}
 \le e^{-I_\mu(u_t)}=e^{-t}.
\]
For the exponential moment, use
$e^{z/2}-1=\frac12\int_0^z e^{t/2}\,dt$ for $z\ge0$.
Then
\begin{align*}
 \E e^{I_\mu(M)/2}
 &=1+\frac12\int_0^\infty e^{t/2}
             \Pp\{I_\mu(M)>t\}\,dt\\
 &\le1+\frac12\int_0^\infty e^{-t/2}\,dt=2.
 \qedhere
\end{align*}
\end{proof}

For a box $Q\subseteq[n]^k$, define its edge count by
\begin{equation}\label{eq:edge-count}
 e(Q)=\sum_{\mathbf i\in Q}t_{\mathbf i}.
\end{equation}
Then $\E e(Q)=p|Q|$.
Fix an integer $L\ge1$. In mode $j$, a
\emph{labeled partial partition} consists of pairwise disjoint nonempty
sets $D_j^\ell\subseteq[n]$, for a subset of labels
$\ell\in\{0,\ldots,L-1\}$.
Each set $D_j^\ell$ is called a \emph{coordinate class} in mode $j$:
its elements are possible values of the index $i_j$.
The class is identified by the pair $(j,\ell)$; classes in different
modes are distinct even if their underlying subsets of $[n]$ coincide.
For a tuple of assigned labels $\mathbf{b}=(\ell_1,\ldots,\ell_k)$,
write $\mathbf{b}_j=\ell_j$ for its $j$th label and define
\[
 Q_{\mathbf{b}}=D_1^{\ell_1}\times\cdots\times D_k^{\ell_k},\qquad
 \mu_{\mathbf{b}}=p|Q_{\mathbf{b}}|,\qquad
 I_{\mathbf{b}}=I_{\mu_{\mathbf{b}}}(e(Q_{\mathbf{b}})).
\]
Thus $\mu_{\mathbf{b}}=\E e(Q_{\mathbf{b}})$, and $I_{\mathbf{b}}$ is
the rate \eqref{eq:rate} of its observed count.
Given a collection $\mathcal B$ of box labels, let
\[
 \mathcal B_j=\{\mathbf{b}_j:\ \mathbf{b}\in\mathcal B\}
\]
be the labels whose classes occur in mode $j$. Write
\[
 h(S)=|S|\log\frac{en}{|S|},\qquad
H(\mathcal B)=\sum_{j=1}^k\sum_{\ell\in\mathcal B_j}h(D_j^\ell).
\]

Here $h(S)$ bounds the combinatorial cost of choosing $S$: for $s=|S|$,
\[
 \binom ns\le\left(\frac{en}{s}\right)^s=e^{h(S)}.
\]
$H(\mathcal B)$ counts each used coordinate class once. For example,
the two boxes $D_1^0\times D_2^0$ and $D_1^0\times D_2^1$ share
$D_1^0$, which contributes $h(D_1^0)$ rather than $2h(D_1^0)$.

The next lemma bounds the sum of the rates by the class cost and
the cost of choosing the box labels.

\begin{lemma}[Multibox discrepancy]\label{lem:local}
Fix an integer $L\ge1$ and $R>0$, and put $K=2(R+k+3)$.
With probability at least
$1-n^{-R}$, simultaneously for every labeled partial partition in every
mode and every nonempty collection $\mathcal B$ of its box labels,
\begin{equation}\label{eq:local-budget}
 \sum_{\mathbf{b}\in\mathcal B}I_{\mathbf{b}}
 \le K\bigl(H(\mathcal B)+|\mathcal B|\log(eL)\bigr).
\end{equation}
\end{lemma}

\smallskip\noindent\emph{Comparison with graph discrepancy.}
For $k=2$, the graph discrepancy estimate used in the
Kahn--Szemer\'edi argument~\cite{FO05,LR15} can be written in the
following form. Simultaneously for all nonempty $S_1,S_2\subseteq[n]$,
with $Q=S_1\times S_2$ and reference mean $\mu=p|Q|$, either
\[
 e(Q)\le C\mu
 \quad\text{or}\quad
 e(Q)\log\frac{e(Q)}{\mu}
 \le C\bigl(h(S_1)+h(S_2)\bigr).
\]
When $e(Q)\ge e^2\mu$, the quantity $e(Q)\log(e(Q)/\mu)$ is comparable
to $I_\mu(e(Q))$, so both estimates use the same large-deviation rate
in this range. The classical bound holds simultaneously over all
rectangles, but controls each rectangle separately.

In contrast, \eqref{eq:local-budget} controls the sum of the rates over
every subcollection from a coordinate partition. Summing separate
rectangle bounds would charge a row or column class each time a
rectangle uses it. The joint bound charges each class once, plus
$|\mathcal B|\log(eL)$ for the box labels. This saving allows us to
sum the heavy-tuple contributions.

\begin{proof}
\emph{Fixed boxes.}
Fix a collection $\mathcal B\subseteq\{0,\ldots,L-1\}^k$ with $m\ge1$ elements.
An assignment $(D_j^\ell)_{1\le j\le k,\,\ell\in\mathcal B_j}$ is valid
if its sets are nonempty subsets of $[n]$ and pairwise disjoint within each mode.
For a fixed valid assignment, distinct label tuples give disjoint
boxes. Thus the counts $e(Q_{\mathbf{b}})\sim\operatorname{Bin}(|Q_{\mathbf{b}}|,p)$ are independent.
Lemma~\ref{lem:rate} therefore gives
\[
 \E\exp\left\{\frac12\sum_{\mathbf{b}\in\mathcal B}I_{\mathbf{b}}\right\}
 =\prod_{\mathbf{b}\in\mathcal B}\E e^{I_{\mathbf{b}}/2}\le2^m.
\]
Markov's inequality gives, for $a>0$,
\[
 \Pp\left\{\sum_{\mathbf{b}\in\mathcal B}I_{\mathbf{b}}>a\right\}
 \le e^{-a/2}\E\exp\left(\frac12\sum_{\mathbf{b}\in\mathcal B}I_{\mathbf{b}}\right)
 \le2^m e^{-a/2}.
\]
Taking $a=K(H(\mathcal B)+m\log(eL))$ yields
\begin{equation}\label{eq:fixed-budget-tail}
 \Pp\left\{\sum_{\mathbf{b}\in\mathcal B}I_{\mathbf{b}}>
       K\bigl(H(\mathcal B)+m\log(eL)\bigr)\right\}
 \le2^m(eL)^{-Km/2}e^{-KH(\mathcal B)/2}.
\end{equation}

\smallskip\noindent\emph{Uniform bound.}
The entropy term accounts for the choice of coordinate sets. Put
\[
 Z=\sum_{\varnothing\ne S\subseteq[n]}e^{-K h(S)/2}.
\]
Using $\binom ns\le(en/s)^s$, the monotonicity of $s\log(en/s)$ on
$[1,n]$, and $K/2-1=R+k+2$, we obtain
\begin{align}
 Z
 &=\sum_{s=1}^n\binom ns
       \exp\left\{-\frac K2s\log\frac{en}{s}\right\}\notag\\
 &\le\sum_{s=1}^n(en/s)^{-(K/2-1)s}
 \le n(en)^{-(K/2-1)}
 \le n^{-R-1}\le1.
 \label{eq:Z}
\end{align}
For fixed $\mathcal B$, assign one nonempty set to each pair
$(j,\ell)$ with $\ell\in\mathcal B_j$, requiring disjointness within
each mode. Only used labels are assigned sets, and $H(\mathcal B)$
is evaluated at that assignment. Dropping the disjointness requirement
enlarges the nonnegative sum and separates the choices of sets,
so the sum factors:
\begin{align*}
 \sum_{\substack{\varnothing\ne D_j^\ell\subseteq[n]\ (j\in[k],\,\ell\in\mathcal B_j)\\
                  D_j^\ell\cap D_j^{\ell'}=\varnothing\ (j\in[k],\,\ell,\ell'\in\mathcal B_j,\,\ell\ne\ell')}}
       e^{-K H(\mathcal B)/2}
 &\le\prod_{j=1}^k\prod_{\ell\in\mathcal B_j}
       \left(\sum_{\varnothing\ne S\subseteq[n]}e^{-K h(S)/2}\right)\\
 &=Z^{\sum_{j=1}^k|\mathcal B_j|}\le Z.
\end{align*}
The last inequality uses $Z\le1$ and $\sum_j|\mathcal B_j|\ge1$.

Let $\mathcal F$ be the event that some nonempty label collection and
valid assignment violate \eqref{eq:local-budget}. The union bound,
\eqref{eq:fixed-budget-tail}, and the assignment estimate give
\begin{align*}
 \Pp(\mathcal F)
 &\le\sum_{m=1}^{L^k}
     \sum_{\substack{\mathcal B\subseteq\{0,\ldots,L-1\}^k\\|\mathcal B|=m}}
       2^m(eL)^{-Km/2}Z\\
 &=Z\sum_{m=1}^{L^k}\binom{L^k}{m}2^m(eL)^{-Km/2}.
\end{align*}
Using $\binom{L^k}{m}\le L^{km}$ and extending the sum to all positive
integers gives
\[
 \Pp(\mathcal F)
 \le Z\sum_{m=1}^\infty
       \left(2e^{-K/2}L^{k-K/2}\right)^m.
\]
Since $K/2\ge k+3$ and $L\ge1$,
\[
 2e^{-K/2}L^{k-K/2}\le2e^{-(k+3)}<\frac12.
\]
The geometric series therefore yields
\[
 \Pp(\mathcal F)\le Z\sum_{m=1}^\infty2^{-m}=Z\le n^{-R}.
\]
This proves \eqref{eq:local-budget} simultaneously for all choices.
\end{proof}

\subsubsection{The deterministic heavy-tuple bound}\label{sec:heavy-sum}

Once the fiber-degree and discrepancy bounds hold, the remaining
argument is deterministic and uniform over the unit vectors. Fiber
degrees control boxes with one small amplitude, dyadic masses
control boxes with moderate counts, and joint discrepancy controls
the remaining boxes.

For unit vectors $x_1,\ldots,x_k$, write $x=(x_1,\ldots,x_k)$ and
define the heavy index set at threshold $\tau=\sqrt d/n$ by
\[
 \mathcal H(x)=\left\{\mathbf i\in[n]^k:
                  \prod_{j=1}^k|x_{j,i_j}|>\tau\right\}.
\]
Thus $\mathcal H(x)=[n]^k\setminus\mathcal L(x)$.

We first bound the entropy cost independently of the number of
dyadic levels.

\begin{lemma}[Entropy of dyadic masses]\label{lem:mass}
Let $n\ge2$ and $L\ge1$ be integers, and set
\[
 t_\ell=n4^{-\ell},\qquad \ell=0,\ldots,L-1.
\]
If $\alpha_\ell\ge0$ and $\sum_{\ell=0}^{L-1}\alpha_\ell\le4$, then
\begin{equation}\label{eq:mass-entropy}
 \sum_{\substack{0\le\ell<L\\t_\ell>1}}\alpha_\ell
 \left(1+\frac{\log_+(1/\alpha_\ell)}{\log(et_\ell)}\right)
 \le 8+\frac43.
\end{equation}
Terms with $\alpha_\ell=0$ are interpreted as zero.
\end{lemma}
\begin{proof}
We bound each term by twice its mass and a geometrically summable
error.
For $t_\ell>1$ and $\alpha_\ell>0$,
$\log_+u\le u$ gives
\[
 \log_+(1/\alpha_\ell)
 \le\log t_\ell+\log_+\!\left(\frac1{\alpha_\ell t_\ell}\right)
 \le\log t_\ell+\frac1{\alpha_\ell t_\ell}.
\]
Since $\log(et_\ell)\ge1$, it follows that
\begin{equation}\label{eq:mass-one-term}
 \alpha_\ell\left(1+
       \frac{\log_+(1/\alpha_\ell)}{\log(et_\ell)}\right)
 \le2\alpha_\ell+\frac1{t_\ell}.
\end{equation}
The same bound holds at $\alpha_\ell=0$ by convention.
The levels with $t_\ell>1$ are $0,\ldots,\ell_*$, where
$\ell_*:=\max\{0\le\ell<L:4^\ell<n\}$ exists since $n\ge2$.
The largest reciprocal is less than one; the preceding reciprocals
decrease by a factor of four. Thus
\[
 \sum_{\substack{0\le\ell<L\\t_\ell>1}}\frac1{t_\ell}
 =\frac1n\sum_{\ell=0}^{\ell_*}4^\ell
 \le\frac{4^{\ell_*}}n\sum_{q=0}^{\infty}4^{-q}
 <\frac43.
\]
Summing \eqref{eq:mass-one-term} and using
$\sum_\ell\alpha_\ell\le4$
proves \eqref{eq:mass-entropy}.
\end{proof}

\begin{proposition}[Deterministic multilevel summation]\label{prop:heavy}
Fix $k\ge2$ and $\kappa,K>0$. Let $n\ge2$, $1\le D<n^2$, and
$\tau=\sqrt D/n$. Here $D$ is a reference degree, not necessarily
an expected fiber degree. Let
\begin{equation}\label{eq:L}
 L=1+\left\lfloor\log_2(1/\tau)\right\rfloor.
\end{equation}
Let $T$ be a tensor with nonnegative entries.
Suppose that $T$ satisfies both of the following properties:
\begin{enumerate}[label=(\roman*),leftmargin=*]
\item every one-dimensional fiber has degree at most $\kappa D$;
\item for every labeled partial partition with labels in $\{0,\ldots,L-1\}$
      in each mode, and every nonempty subcollection of its boxes, the local
      entropy bound \eqref{eq:local-budget} holds with constant $K$,
      using $I_{\mathbf{b}}=I_{\mu_{\mathbf{b}}}(e(Q_{\mathbf{b}}))$ and $\mu_{\mathbf{b}}=(D/n)|Q_{\mathbf{b}}|$.
\end{enumerate}
Put $\mathcal H_D(x)=\{\mathbf i:\prod_j|x_{j,i_j}|>\tau\}$.
Then there is a constant $C_{k,\kappa,K}$ such that
\begin{equation}\label{eq:allp-heavy}
 \sup_{\|x_1\|_2=\cdots=\|x_k\|_2=1}
 \sum_{\mathbf i\in\mathcal H_D(x)}
       t_{\mathbf i}\prod_{j=1}^k|x_{j,i_j}|
 \le C_{k,\kappa,K}\sqrt D.
\end{equation}
No lower bound on $D/\log n$ is required.
\end{proposition}
\begin{proof}
Fix $T$ satisfying \textup{(i)}--\textup{(ii)} and unit vectors
$x_1,\ldots,x_k$.

\subsubsection*{Dyadic decomposition}
Put $a_\ell=2^{-\ell}$ and, for $0\le\ell<L$, let
\[
 D_j^\ell=\{i:\ |x_{j,i}|>\tau,\quad a_\ell/2<|x_{j,i}|\le a_\ell\}.
\]
Since $2^{-L}<\tau\le2^{-(L-1)}$, the nonempty classes, indexed by
their original levels, form a partial partition of $\{i:|x_{j,i}|>\tau\}$.
Define
\begin{equation}\label{eq:alpha}
 \alpha_{j,\ell}=|D_j^\ell|a_\ell^2,\qquad
 \sum_{\ell=0}^{L-1}\alpha_{j,\ell}
 \le4\sum_{\ell=0}^{L-1}\sum_{i\in D_j^\ell}|x_{j,i}|^2
 \le4.
\end{equation}
Unless stated otherwise, level sums below range over $0,\ldots,L-1$;
empty classes have mass zero.

For each nonempty box $Q_{\mathbf b}=D_1^{\ell_1}\times\cdots\times D_k^{\ell_k}$,
write $a_j=a_{\ell_j}$, $s_j=|D_j^{\ell_j}|$, and
$\alpha_j=\alpha_{j,\ell_j}$. Set
\begin{equation}\label{eq:box-parameters}
 a_{\mathbf b}=\prod_j a_j,\qquad
 \mu_{\mathbf b}=\frac Dn\prod_j s_j>0,\qquad
 \rho_{\mathbf b}=\frac{a_{\mathbf b}}{\tau},\qquad
 \lambda_{\mathbf b}=\frac{e(Q_{\mathbf b})}{\mu_{\mathbf b}}.
\end{equation}
Thus $\rho_{\mathbf b}$ measures the dyadic box amplitude relative to
the heavy threshold, while $\lambda_{\mathbf b}$ compares the observed
count with the reference mean.
Let $\mathcal B$ index the nonempty boxes with $a_{\mathbf b}>\tau$.
Every heavy tuple has $|x_{j,i_j}|>\tau$ in each mode, since
coordinate magnitudes are at most one. The disjoint dyadic classes
therefore assign it to a unique box with
$a_{\mathbf b}\ge\prod_j|x_{j,i_j}|>\tau$. Nonnegativity gives
\begin{equation}\label{eq:heavy-box-domination}
 \sum_{\mathbf i\in\mathcal H_D(x)}t_{\mathbf i}\prod_j|x_{j,i_j}|
 \le\sum_{\mathbf b\in\mathcal B}a_{\mathbf b}e(Q_{\mathbf b}).
\end{equation}
Since $|Q_{\mathbf b}|a_{\mathbf b}^2=\prod_j\alpha_j$ and
$\tau=\sqrt D/n$, the normalized contribution of each box is
\begin{equation}\label{eq:normalized-box-weight}
 \frac{a_{\mathbf b}e(Q_{\mathbf b})}{\sqrt D}
 =\frac{\lambda_{\mathbf b}}{\rho_{\mathbf b}}\prod_j\alpha_j.
\end{equation}

Following the class-by-class organization of \cite[Section~4.3]{ZZ21},
partition the box labels into
\begin{align*}
 \mathcal C_1
 &=\{\mathbf b\in\mathcal B:\ na_j^2\le\rho_{\mathbf b}
                                      \text{ for some }j\},\\
 \mathcal C_2
 &=\{\mathbf b\in\mathcal B\setminus\mathcal C_1:
                               \lambda_{\mathbf b}\le e^2\rho_{\mathbf b}\},\\
 \mathcal C_3&=\mathcal B\setminus(\mathcal C_1\cup\mathcal C_2).
\end{align*}
Boxes in $\mathcal C_1$ have one coordinate amplitude small relative
to the other amplitudes. Among the remaining boxes, $\mathcal C_2$
has a moderate count relative to its amplitude: the ratio
$\lambda_{\mathbf b}/\rho_{\mathbf b}$ is at most $e^2$.
The remaining dense boxes form $\mathcal C_3$. We bound the three
contributions using fiber degrees, dyadic masses, and joint discrepancy,
respectively.

\subsubsection*{Boxes in $\mathcal C_1$: the fiber-degree bound}
Fix a mode $j$ and the other level labels, and put
$q=(\prod_{i\ne j}a_i)/\sqrt D$. Since
$\rho_{\mathbf b}=na_{\mathbf b}/\sqrt D$, the condition
$na_j^2\le\rho_{\mathbf b}$ is equivalent to $a_j\le q$.
Because the classes $D_j^{\ell_j}$ are disjoint, these boxes occupy
disjoint portions of the same $\prod_{i\ne j}s_i$ mode-$j$ fibers.
By \textup{(i)},
\[
 \sum_{\substack{\ell_j:\ \mathbf b\in\mathcal B\\a_j\le q}}e(Q_{\mathbf b})
 \le\kappa D\prod_{i\ne j}s_i.
\]
Using $a_j\le q$ before summing gives
\[
 \sum_{\substack{\ell_j:\ \mathbf b\in\mathcal B\\a_j\le q}}
       \frac{a_{\mathbf b}e(Q_{\mathbf b})}{\sqrt D}
 \le\frac{q\prod_{i\ne j}a_i}{\sqrt D}\,
       \kappa D\prod_{i\ne j}s_i
 =\kappa\prod_{i\ne j}\alpha_i.
\]
Summing over the other labels and then over $j$ covers every box in
$\mathcal C_1$, possibly more than once. Nonnegativity therefore gives
an upper bound. By \eqref{eq:alpha},
\begin{equation}\label{eq:unbalanced-sum}
 \sum_{\mathbf b\in\mathcal C_1}
       \frac{a_{\mathbf b}e(Q_{\mathbf b})}{\sqrt D}
 \le\kappa\sum_{j=1}^k
       \prod_{i\ne j}\left(\sum_{\ell_i}\alpha_{i,\ell_i}\right)
 \le k\kappa\,4^{k-1}.
\end{equation}

\subsubsection*{Boxes in $\mathcal C_2$: the mass bound}
The identity \eqref{eq:normalized-box-weight} and the defining bound
$\lambda_{\mathbf b}/\rho_{\mathbf b}\le e^2$ give
\[
 \frac{a_{\mathbf b}e(Q_{\mathbf b})}{\sqrt D}
 \le e^2\prod_j\alpha_j
 \qquad(\mathbf b\in\mathcal C_2).
\]
Enlarging the sum to all label tuples makes the nonnegative mass
products factor by mode. Thus \eqref{eq:alpha} gives
\begin{equation}\label{eq:moderate-sum}
 \sum_{\mathbf b\in\mathcal C_2}
       \frac{a_{\mathbf b}e(Q_{\mathbf b})}{\sqrt D}
 \le e^2\prod_{j=1}^k\left(\sum_{\ell_j}\alpha_{j,\ell_j}\right)
 \le e^2 4^k.
\end{equation}

\subsubsection*{Boxes in $\mathcal C_3$: the joint discrepancy bound}
Assume $\mathcal C_3\ne\varnothing$. Its defining conditions are
\begin{equation}\label{eq:high}
 1<\rho_{\mathbf b}<na_j^2\quad(j\in[k]),\qquad
 \lambda_{\mathbf b}>e^2\rho_{\mathbf b}.
\end{equation}
Hence \eqref{eq:rate} yields
\[
 I_{\mathbf b}\ge e(Q_{\mathbf b})(\log\lambda_{\mathbf b}-1)
               \ge e(Q_{\mathbf b})\log(e\rho_{\mathbf b}).
\]
For $s\ge0$, put
\[
 \mathcal C_3(s)=\left\{\mathbf b\in\mathcal C_3:
        s<\frac{a_{\mathbf b}}{\sqrt D\log(e\rho_{\mathbf b})}\right\},
 \qquad H(\varnothing)=0.
\]
The thresholded collections let us apply the discrepancy bound
before summing the weighted counts. Each box is included for an interval
of length $a_{\mathbf b}/(\sqrt D\log(e\rho_{\mathbf b}))$.
Tonelli's theorem and \textup{(ii)} give
\begin{align}
 \sum_{\mathbf b\in\mathcal C_3}\frac{a_{\mathbf b}e(Q_{\mathbf b})}{\sqrt D}
 &\le\int_0^\infty\sum_{\mathbf b\in\mathcal C_3(s)}I_{\mathbf b}\,ds
 \notag\\
 &\le K\int_0^\infty H(\mathcal C_3(s))\,ds
       +K\log(eL)\int_0^\infty|\mathcal C_3(s)|\,ds.
 \label{eq:high-integrated}
\end{align}
Assumption \textup{(ii)} holds for every subcollection, including
$\mathcal C_3(s)$, even though it depends on the counts.

For the entropy integral, if $D_j^\ell$ occurs in $\mathcal C_3(s)$,
some box using it satisfies \eqref{eq:high}. The monotonicity of
$u/\log(eu)$ on $[1,\infty)$ gives
\[
 s<\frac{\rho_{\mathbf b}}{n\log(e\rho_{\mathbf b})}
 \le\frac{a_\ell^2}{\log(en a_\ell^2)},\qquad na_\ell^2>1.
\]
At each threshold, the entropy counts each used class once,
regardless of how many boxes use it. The thresholds at which the
class occurs form an interval whose length is bounded above by the
preceding estimate. Its integrated entropy cost is therefore at most
\[
 \frac{h(D_j^\ell)a_\ell^2}{\log(en a_\ell^2)}
 =\frac{\alpha_{j,\ell}\log(en a_\ell^2/\alpha_{j,\ell})}
          {\log(en a_\ell^2)}
 \le\alpha_{j,\ell}\left(1+
       \frac{\log_+(1/\alpha_{j,\ell})}{\log(en a_\ell^2)}\right).
\]
Classes unused by these boxes never appear in the thresholded
collections and contribute nothing to the entropy integral. Summing
over classes, we apply Lemma~\ref{lem:mass} in each mode with $t_\ell=na_\ell^2=n4^{-\ell}$
and the masses in \eqref{eq:alpha}:
\begin{equation}\label{eq:total-entropy-charge}
 \int_0^\infty H(\mathcal C_3(s))\,ds
 \le\sum_{j=1}^k\sum_{\substack{0\le\ell<L\\na_\ell^2>1}}
       \alpha_{j,\ell}\left(1+
       \frac{\log_+(1/\alpha_{j,\ell})}{\log(en a_\ell^2)}\right)
 \le\frac{28}{3}k.
\end{equation}

For the label integral, we seek a factor $1/\log(e/\tau)$ to absorb
$\log(eL)$ in \eqref{eq:high-integrated}. Put $v=\log(1/\tau)>0$ and,
for a heavy box with level tuple $(\ell_1,\ldots,\ell_k)$, set
$u=(\log2)\sum_j\ell_j$. Then $0\le u<v$ and
$\log(e\rho_{\mathbf b})=1+v-u$. The identity
\[
 (1+u)(1+v-u)=1+v+u(v-u)\ge1+v
\]
therefore gives
\[
 \frac1{\log(e\rho_{\mathbf b})}
 \le\frac{1+u}{\log(e/\tau)}.
\]
Consequently,
\begin{align}
 \int_0^\infty|\mathcal C_3(s)|\,ds
 &=\sum_{\mathbf b\in\mathcal C_3}
       \frac{a_{\mathbf b}}{\sqrt D\log(e\rho_{\mathbf b})}
 \label{eq:overhead-sum}\\
 &\le\frac1{\sqrt D\log(e/\tau)}
   \sum_{\ell_1,\ldots,\ell_k\ge0}
   2^{-\sum_j\ell_j}\left(1+(\log2)\sum_j\ell_j\right)
 \notag\\
 &=\frac{2^k(1+k\log2)}{\sqrt D\log(e/\tau)}.
 \label{eq:allp-label}
\end{align}
The last equality uses
$\sum_{\ell\ge0}2^{-\ell}=\sum_{\ell\ge0}\ell2^{-\ell}=2$.
Since $L\le2^{L-1}\le1/\tau$, we have
$\log(eL)\le\log(e/\tau)$. Thus the label term in
\eqref{eq:high-integrated} is bounded independently of the number of levels.
Together with \eqref{eq:total-entropy-charge}, this gives
\begin{equation}\label{eq:high-sum}
 \sum_{\mathbf b\in\mathcal C_3}
       \frac{a_{\mathbf b}e(Q_{\mathbf b})}{\sqrt D}
 \le K\left\{\frac{28k}{3}
       +\frac{2^k(1+k\log2)}{\sqrt D}\right\}
 \le K\left\{\frac{28k}{3}+2^k(1+k\log2)\right\}.
\end{equation}

Combining \eqref{eq:unbalanced-sum}, \eqref{eq:moderate-sum}, and
\eqref{eq:high-sum} with \eqref{eq:heavy-box-domination} proves the
bound uniformly over the unit vectors, with
\[
 C_{k,\kappa,K}
 =k\kappa4^{k-1}+e^2 4^k
  +K\left\{\frac{28k}{3}+2^k(1+k\log2)\right\}.
\]
\end{proof}

\subsection{Completion of the proof}\label{sec:completion}

It remains to subtract the mean from the heavy sum, combine the
light and heavy estimates on the product net, and recover the norm.

\begin{proof}[Proof of Theorem~\ref{thm:main}]
By increasing $C_{k,r,c}$ if necessary,
we may assume that $n$ is sufficiently large that
$d\ge c\log n\ge1$. \reviewblue{Set $R=r+2$ and take $D=d$ in
Proposition~\ref{prop:heavy}, so $\tau=\sqrt d/n$ and $L$ is given
by \eqref{eq:L}.}
Apply Lemma~\ref{lem:local} with the deterministic $L$ in
\eqref{eq:L}, Lemma~\ref{lem:fibers}, and Lemma~\ref{lem:light}.
The three events hold simultaneously with failure probability at most
$3n^{-(r+2)}\le n^{-r}$.
On their intersection, fix unit vectors $x_1,\ldots,x_k$ and put
$q_{\mathbf i}=\prod_j|x_{j,i_j}|$.
\reviewblue{Proposition~\ref{prop:heavy}, with $D=d=np$, gives}
\[
 \sum_{\mathbf i\in\mathcal H(x)}t_{\mathbf i}q_{\mathbf i}
 \le C_{\mathrm h}\sqrt d,
\]
uniformly in the unit vectors, with $C_{\mathrm h}$ depending only on
$k,r,c$.

Since $P=p\,\one^{\otimes k}$, centering costs at most
$p\sum_{\mathbf i\in\mathcal H(x)}q_{\mathbf i}$. For heavy tuples,
$q_{\mathbf i}>\tau$ implies $q_{\mathbf i}\le q_{\mathbf i}^2/\tau$, so
\begin{equation}\label{eq:heavy-mean}
\begin{aligned}
 p\sum_{\mathbf i\in\mathcal H(x)}q_{\mathbf i}
 &\le\frac p\tau\sum_{\mathbf i\in\mathcal H(x)}q_{\mathbf i}^2
 \le\frac p\tau\sum_{\mathbf i\in[n]^k}q_{\mathbf i}^2\\
 &=\frac p\tau\prod_{j=1}^k\|x_j\|_2^2
 =\frac{np}{\sqrt d}=\sqrt d.
\end{aligned}
\end{equation}
For $x\in\mathcal T^k$, split $W=T-P$ into light and heavy tuples.
The triangle inequality and $|t_{\mathbf i}-p|\le t_{\mathbf i}+p$ give
\begin{align*}
 |W(x_1,\ldots,x_k)|
 &\le |L_W(x)|
       +\sum_{\mathbf i\in\mathcal H(x)}t_{\mathbf i}q_{\mathbf i}
       +p\sum_{\mathbf i\in\mathcal H(x)}q_{\mathbf i}\\
 &\le (B_{k,R}+C_{\mathrm h}+1)\sqrt d,
\end{align*}
by Lemma~\ref{lem:light}, the heavy bound, and \eqref{eq:heavy-mean}.
Finally, applying \eqref{eq:net-bound} to the full multilinear form gives
\[
 \|W\|\le
       2\max_{(y_1,\ldots,y_k)\in\mathcal T^k}
            |W(y_1,\ldots,y_k)|\le C_{k,r,c}\sqrt d.
\]
This proves \eqref{eq:main}.
\end{proof}

\section{Proof of Theorem~\ref{thm:homogeneous-allp}}
\label{sec:homogeneous-allp}

\reviewblue{For the upper bound, we apply Proposition~\ref{prop:heavy}
with reference degree $D=d_{\mathrm{eff}}$. This changes the threshold
and reference box means, not the observed tensor.
A single large fiber gives the matching lower bound. We then derive
the expectation bounds and describe the very sparse transition.}
\reviewblue{The same upper-bound argument also proves
Corollary~\ref{cor:weighted}.}

\subsection{Upper bound: increase the reference degree, not the tensor}

\begin{proof}[Proof of \eqref{eq:hom-upper-tail}]
\reviewblue{We prove the upper bound for the full range $0<p\le1$.}
Set $D=d_{\mathrm{eff}}$, $R=r+2$, and $\tau=\sqrt D/n$.
Because $D\ge d$, this reference degree applies to both the light and
heavy parts. \reviewblue{Since $\log(e+\log n/d)\ge1$ and $d\le n$,}

\[
 1\le D\le n+\log n<n^2,
\]

so $0<\tau<1$.

For a fiber degree $Y\sim\operatorname{Bin}(n,p)$ and $u>0$, put
$\eta=\log(e+u/d)\ge1$. The exponential-moment bound
$\E e^{\eta Y}\le\exp\{d(e^\eta-1)\}$ yields
\begin{equation}\label{eq:hom-fiber-chernoff}
 \Pp\left\{Y\ge2\left(d+
                 \frac{u}{\log(e+u/d)}\right)\right\}
 \le\exp\{-u-d(2\eta-e+1)\}\le e^{-u}.
\end{equation}
Choose $B=R+k+\lceil\log_2 k\rceil$ and $u=B\log n$.
Because $B\ge1$,
\[
 d+\frac{B\log n}{\log(e+B\log n/d)}
 \le B\left(d+\frac{\log n}{\log(e+\log n/d)}\right)
 \le BD.
\]
A union bound over the $kn^{k-1}$ fibers shows that all degrees are
at most $2BD$ with failure probability at most $n^{-R}$.

We verify discrepancy at the same reference degree by applying
Lemma~\ref{lem:local} with
$L=1+\lfloor\log_2(1/\tau)\rfloor$ and failure exponent $R$.
Since $I_\mu(u)$ is nonincreasing in $\mu$ and $D/n\ge p$, its
multibox inequality for $p|Q|$ also holds for $(D/n)|Q|$.
Together with the fiber bound, \reviewblue{Proposition~\ref{prop:heavy}} gives,
uniformly over unit vectors,
\begin{equation}\label{eq:hom-heavy-reference}
 \sum_{\mathbf i:\,\prod_j|x_{j,i_j}|>\tau}
       t_{\mathbf i}\prod_j|x_{j,i_j}|
 \le C_{k,R}\sqrt D.
\end{equation}
The heavy contribution of the mean is bounded by
\[
 p\sum_{\mathbf i:\,\prod_j|x_{j,i_j}|>\tau}
        \prod_j|x_{j,i_j}|
 \le\frac p\tau\sum_{\mathbf i}\prod_jx_{j,i_j}^2
 =\frac d{\sqrt D}\le\sqrt D.
\]

For each product-net direction, the light sum is centered, with
summands bounded by $\tau$ and total variance at most $p\le D/n$.
The Bernstein and net argument in Lemma~\ref{lem:light}, applied with
these bounds and reference degree $D$, controls all light sums by
$B_{k,R}\sqrt D$, with failure probability at most $n^{-R}$.
Adding the light, heavy, and mean contributions and applying
\eqref{eq:net-bound} proves the norm bound. The three exceptional events
have total probability at most $3n^{-R}\le n^{-r}$.
\end{proof}

\begin{proof}[Proof of Corollary~\ref{cor:weighted}]
Use $n,p,d$ as in the corollary. The case $A=0$ is immediate.
By homogeneity, assume $\|A\|_\infty=1$.
Pad $A$ and $\Omega$ with zero entries to obtain
$n\times\cdots\times n$ tensors, setting the added sampling
probabilities to zero. Padding preserves the centered spectral norm,
$\|A\|_\infty$, and $\sum_{\mathbf i}p_{\mathbf i}$;
the same upper bound $p$ remains valid.
Put $D=d_{\mathrm{eff}}$, $\tau=\sqrt D/n$, and $R=r+2$. Since $p\le1$,
\[
 1\le D\le n+\log n<n^2.
\]

Couple $\Omega$ to a homogeneous Bernoulli$(p)$ tensor
$\Omega^*$ using independent uniform random variables:
\[
 \omega_{\mathbf i}=\one\{U_{\mathbf i}\le p_{\mathbf i}\},
 \qquad
 \omega^*_{\mathbf i}=\one\{U_{\mathbf i}\le p\}.
\]
Then $\Omega\le\Omega^*$ entrywise. Apply the uniform homogeneous
heavy bound \eqref{eq:hom-heavy-reference} to $\Omega^*$.
Since all summands are nonnegative, uniformly over unit vectors,
\[
 \sum_{\mathbf i:q_{\mathbf i}>\tau}
       \omega_{\mathbf i}q_{\mathbf i}
 \le \sum_{\mathbf i:q_{\mathbf i}>\tau}
       \omega^*_{\mathbf i}q_{\mathbf i}
 \le C_{k,R}\sqrt D,
 \qquad q_{\mathbf i}=\prod_j|x_{j,i_j}|,
\]
outside an event of probability at most $2n^{-R}$.

For each product-net direction, the weighted light sum has centered
independent summands bounded by $\tau$, with total variance at most
\[
 \sum_{\mathbf i}p_{\mathbf i}a_{\mathbf i}^2
       \prod_jx_{j,i_j}^2
 \le p\le D/n.
\]
The Bernstein and net argument of Lemma~\ref{lem:light} therefore
bounds all light sums by $C_{k,R}\sqrt D$, with failure probability
at most $n^{-R}$. Since $|a_{\mathbf i}|\le1$, the weighted heavy
sum is bounded in absolute value by the count contribution above
plus the mean contribution
\[
 \sum_{\mathbf i:q_{\mathbf i}>\tau}p_{\mathbf i}q_{\mathbf i}
 \le\frac{p}{\tau}\sum_{\mathbf i}q_{\mathbf i}^2
 =\frac{d}{\sqrt D}\le\sqrt D.
\]
Combining the light and heavy bounds and applying
\eqref{eq:net-bound} proves \eqref{eq:weighted}, since
$3n^{-R}\le n^{-r}$.

For the expectation, use \eqref{eq:weighted} with a fixed
$r>k/2$ and the deterministic bound
$\|A\odot(\Omega-\E\Omega)\|\le n^{k/2}$ to obtain
$\E\|A\odot(\Omega-\E\Omega)\|\le C_k\sqrt D$.
On the other hand, a tensor with a single unit entry has norm one,
so the triangle inequality gives
\[
 \E\|A\odot(\Omega-\E\Omega)\|
 \le\sum_{\mathbf i}|a_{\mathbf i}|
             \E|\omega_{\mathbf i}-p_{\mathbf i}|
 =2\sum_{\mathbf i}|a_{\mathbf i}|
             p_{\mathbf i}(1-p_{\mathbf i})
 \le2\sum_{\mathbf i}p_{\mathbf i}.
\]
Taking the smaller bound and restoring $\|A\|_\infty$ proves
\eqref{eq:weighted-expectation}.
\end{proof}

\subsection{Lower bound: one large fiber already forces a large norm}

\begin{proof}[Proof of \eqref{eq:hom-lower-tail}]
\reviewblue{Assume $0<p\le1/2$ as in the lower-bound statement.}
A fiber norm bounds the tensor norm from below: fix the other $k-1$
vectors to coordinate vectors and optimize along the fiber.
Fix one direction, and let
$\Delta$ be the maximum of its $n^{k-1}$ independent
$\operatorname{Bin}(n,p)$ fiber degrees. A centered fiber of degree
$Y$ has squared Euclidean norm
\begin{equation}\label{eq:hom-fiber-lower}
 Y(1-p)^2+(n-Y)p^2=Y(1-2p)+np^2.
\end{equation}
For $1/4\le p\le1/2$, every centered entry has magnitude at least
$1/4$, so the norm is at least $\sqrt n/4\ge\sqrt{d_{\mathrm{eff}}}/4$.
For $0<p<1/4$, \eqref{eq:hom-fiber-lower} gives
$\|T-\reviewblue{P}\|\ge\sqrt{\Delta/2}$.

By a Chernoff bound, at least $n^kp/2$ entries are nonzero except
with probability $\exp(-n^kp/8)$. On this event the count is positive
and the average fiber degree is at least $d/2$, so
\begin{equation}\label{eq:hom-baseline-degree}
 \Delta\ge\max\{1,d/2\}.
\end{equation}
This controls the baseline scale. To capture the additional
contribution of the largest fiber in
\eqref{eq:hom-effective-degree}, put
\[
 \zeta=\frac{\log n}{\log(e+\log n/d)}.
\]
If $d>\log n/8$, use $\zeta\le\log n<8d$; if $\zeta<16$,
use $d+\zeta\le2\Delta+16\Delta$. In either case,
\eqref{eq:hom-baseline-degree} gives $d_{\mathrm{eff}}\le18\Delta$.

It remains to treat $d\le\log n/8$ and $\zeta\ge16$. Here we use
independence across fibers to find an unusually large degree.
Let
$v=\lfloor \zeta/8\rfloor$. Then $\zeta/16\le v\le \zeta/8$ and $v\le n$.
Using $p\le1/2$ and $\binom nv\ge(n/v)^v$, one fiber satisfies
\begin{align*}
 \Pp\{Y=v\}
 &\ge (d/v)^v(1-p)^n\\
 &\ge\exp\{-v\log(v/d)-2d\}\\
 &\ge\exp\{-3\log n/8\}=n^{-3/8}.
\end{align*}
The last step uses
$\log_+(v/d)\le\log(e+\log n/d)$,
$v\le \zeta/8$, and $2d\le\log n/4$.
Independence across the $n^{k-1}$ fibers therefore gives
\begin{equation}\label{eq:hom-large-fiber}
 \Pp\{\Delta<v\}
 \le\exp\{-n^{k-1-3/8}\}
 \le\exp\{-n^{5/8}\}.
\end{equation}
Outside this event, $\Delta\ge v\ge\zeta/16$. Together with
\eqref{eq:hom-baseline-degree}, this again gives
$d_{\mathrm{eff}}\le18\Delta$, and hence the norm is at least
$\sqrt{d_{\mathrm{eff}}}/6$ outside the two stated exceptional events.
\end{proof}

\subsection{Expectation and the very sparse transition}

\begin{proof}[Proof of \eqref{eq:hom-expectation}]
\reviewblue{For the upper bound, let $0<p\le1$.}
The deterministic estimate
$\|T-\reviewblue{P}\|\le n^{k/2}$ and
\eqref{eq:hom-upper-tail}, with a fixed $r>k/2$, give
\[
 \E\|T-\reviewblue{P}\|\le C_k\sqrt{d_{\mathrm{eff}}}.
\]
Each tensor with a single unit entry has norm one, so also
\[
 \E\|T-\reviewblue{P}\|
 \le\E\sum_{\mathbf i}t_{\mathbf i}+p n^{k/2}
 \le2n^kp.
\]
Taking the smaller of these two estimates proves the upper bound.

\reviewblue{For the lower bound, assume $0<p\le1/2$.}
If $n^kp\le1$, the probability of at least one nonzero entry is
at least $1-e^{-n^kp}\ge(1-e^{-1})n^kp$.
On that event, testing the corresponding coordinate vectors gives
$\|T-\reviewblue{P}\|\ge1-p\ge1/2$.
Thus one nonzero entry suffices for the lower bound when $n^kp\le1$.

For $n^kp\ge1$, the case $p\ge1/4$ follows from the deterministic
fiber bound. If $p<1/4$, \eqref{eq:hom-baseline-degree} holds with
probability at least $1-e^{-1/8}$, giving
\[
 \E\|T-\reviewblue{P}\|\ge c\sqrt{\max\{1,d\}}.
\]
Put $\zeta=\log n/\log(e+\log n/d)$. As in the preceding proof, the
baseline bound controls $\sqrt{d_{\mathrm{eff}}}$ if $d>\log n/8$ or
$\zeta<16$.
Otherwise, \eqref{eq:hom-large-fiber} also gives
\[
 \E\|T-\reviewblue{P}\|
 \ge (1-e^{-2^{5/8}})\sqrt{\zeta/32}.
\]
The larger of these two bounds controls
$\sqrt{\max\{1,d+\zeta\}}=\sqrt{d_{\mathrm{eff}}}$ up to an absolute
constant, completing the proof.
\end{proof}

At the critical sparsity $n^kp\to\lambda\in(0,\infty)$, the probability
of an empty tensor tends to $e^{-\lambda}$. On this event the centered
norm is $pn^{k/2}\to0$, whereas on the complementary event it is at
least $1-p\to1$. Thus no single deterministic scale gives a two-sided
bound with fixed positive constants and probability tending to one.

\section{Proof of Theorem~\ref{thm:fw}}\label{sec:fw-proof}

Throughout this section, $X,d,P$ are as in \eqref{eq:fw-model}.
Each entry $X_{\mathbf i}$ counts sampled edges with multiplicity,
so its mean $d/n$ may exceed one.
Write $D$ for the effective degree in \eqref{eq:hom-effective-degree},
with $d=m/n^{k-1}$, and use $D\ge\max\{1,d\}$ as the reference
degree in Proposition~\ref{prop:heavy} when $D<n^2$; when $D\ge n^2$,
every tuple is light. The light--heavy decomposition
is unchanged, but independence holds for the edge draws rather than
the tensor entries. We control the light part through the draws and
verify multibox discrepancy through a joint exponential-moment bound.

\subsection{A joint rate estimate for multinomial counts}

A sampled edge belongs to at most one box in a disjoint collection, so
the box counts are multinomial. For nonnegative parameters, their joint
exponential moment is bounded by that of independent Poisson variables.
This comparison supplies the discrepancy argument without entrywise
independence, even when the reference means exceed the actual means.

\begin{lemma}[Joint exponential moment]\label{lem:multinomial}
Let $Q_1,\ldots,Q_s$ be pairwise disjoint nonempty subsets of $[n]^k$,
and write $e(Q_b)=\sum_{\mathbf i\in Q_b}X_{\mathbf i}$ for their
occupancy counts under \eqref{eq:fw-model}.
Choose arbitrary reference means
\[
 \mu_b\ge \frac{m|Q_b|}{n^k}>0,
 \qquad I_b=I_{\mu_b}(e(Q_b)).
\]
Then
\begin{equation}\label{eq:mixed-mgf}
 \E\exp\left(\frac12\sum_{b=1}^s I_b\right)\le2^s.
\end{equation}
\end{lemma}
\begin{proof}
Put $q_b=|Q_b|/n^k$. Disjointness, $1+x\le e^x$, and
$mq_b\le\mu_b$ give, for $\theta_1,\ldots,\theta_s\ge0$,
\begin{align*}
 \E\exp\left(\sum_{b=1}^s\theta_be(Q_b)\right)
 &=\left(1+\sum_{b=1}^s q_b(e^{\theta_b}-1)\right)^m\\
 &\le\exp\left(\sum_{b=1}^s\mu_b(e^{\theta_b}-1)\right).
\end{align*}
For $B\subseteq[s]$ and $t_b\ge0$, choose $u_b\ge\mu_b$ with
$I_{\mu_b}(u_b)=t_b$, as in Lemma~\ref{lem:rate}.
Since $\{I_b>t_b\}=\{e(Q_b)>u_b\}$, Markov's inequality with
$\theta_b=\log(u_b/\mu_b)$ on $B$ and zero elsewhere yields
\[
 \Pp\{I_b>t_b\text{ for every }b\in B\}
 \le \exp\left(\sum_{b\in B}
       [\mu_b(e^{\theta_b}-1)-\theta_bu_b]\right)
 =e^{-\sum_{b\in B}t_b}.
\]
To obtain the rate moment, expand
$\prod_{b=1}^s[1+(e^{I_b/2}-1)]$ and use
$e^{z/2}-1=\frac12\int_0^\infty e^{t/2}\one\{z>t\}\,dt$ for $z\ge0$.
Tonelli's theorem and the joint tail bound give
\begin{align*}
 \E\exp\left(\frac12\sum_{b=1}^s I_b\right)
 &\le\sum_{B\subseteq[s]}\prod_{b\in B}
 \left(\frac12\int_0^\infty e^{-t/2}\,dt\right)
 =\sum_{B\subseteq[s]}1=2^s.
 \qedhere
\end{align*}
\end{proof}

\begin{corollary}[Multinomial multibox discrepancy]\label{cor:multi-local}
Fix a deterministic integer $L\ge1$ and $R>0$.
For the model \eqref{eq:fw-model}, define, for each nonempty box
$Q_{\mathbf{b}}$ of a fixed labeled partial partition,
\[
 e(Q_{\mathbf{b}})=\sum_{\mathbf i\in Q_{\mathbf{b}}}X_{\mathbf i},
 \qquad \mu_{\mathbf{b}}=\E e(Q_{\mathbf{b}})=\frac dn|Q_{\mathbf{b}}|,
 \qquad I_{\mathbf{b}}=I_{\mu_{\mathbf{b}}}(e(Q_{\mathbf{b}})).
\]
Here $e(Q_{\mathbf{b}})$ counts sampled edges in $Q_{\mathbf{b}}$ with
multiplicity, and $I_{\mathbf{b}}$ measures its upper-tail deviation
from the mean $\mu_{\mathbf{b}}$ using the rate function
\eqref{eq:rate}. In particular, $I_{\mathbf{b}}=0$ when
$e(Q_{\mathbf{b}})\le\mu_{\mathbf{b}}$.
With probability at least $1-n^{-R}$, every partial
partition and every nonempty subcollection of its boxes satisfy
\[
 \sum_{\mathbf{b}\in\mathcal B}I_{\mathbf{b}}
 \le K\bigl(H(\mathcal B)+|\mathcal B|\log(eL)\bigr),
 \qquad K=2(R+k+3).
\]
\end{corollary}
\begin{proof}
For each fixed assignment of coordinate classes, the boxes are
disjoint, with means $(d/n)|Q_{\mathbf{b}}|$. Lemma~\ref{lem:multinomial}
replaces the independence step in the proof of
\eqref{eq:fixed-budget-tail}. All subsequent counting, including
\eqref{eq:Z}, is deterministic. The same union bound therefore gives
failure probability at most $n^{-R}$.
\end{proof}

\subsection{Degree and light estimates}

Every fiber has the same degree distribution:
\[
 Y\sim\operatorname{Bin}(m,n^{-(k-1)}),\qquad \E Y=d.
\]
The exponential-moment proof of \eqref{eq:hom-fiber-chernoff}
applies to any binomial variable with mean $d$.
For $R>0$, set $B_{\mathrm f}=R+k+\lceil\log_2k\rceil$ and
$u=B_{\mathrm f}\log n$. The same effective-degree comparison and
union bound over $kn^{k-1}$ fibers therefore give
\[
 \Pp\{\text{some fiber degree exceeds }\kappa D\}
 \le kn^{k-1-B_{\mathrm f}}\le n^{-R},
 \qquad \kappa=2B_{\mathrm f}.
\]
Only the marginal degree distribution is needed here, not independence
between fibers.

For the light part, write the sum over independent edge draws.
Put $W=X-P$ and $\tau=\sqrt D/n$, fix unit vectors
$x_1,\ldots,x_k$, and define
\[
 c_{\mathbf i}=\prod_{j=1}^k x_{j,i_j}\,
                   \one\left\{\left|\prod_{j=1}^k x_{j,i_j}\right|
                                    \le\tau\right\}.
\]
The centered light sum is
\begin{equation}\label{eq:sample-light}
 L_W(x):=\sum_{\mathbf i}c_{\mathbf i}(X_{\mathbf i}-d/n)
 =\sum_{\nu=1}^m\left(c_{E_\nu}-\E c_{E_\nu}\right).
\end{equation}
The summands on the right are independent and bounded in absolute value by
$2\tau$, and their total variance is at most
\begin{equation}\label{eq:sample-variance}
 m\,\E c_{E_1}^2
 =\frac{m}{n^k}\sum_{\mathbf i}c_{\mathbf i}^2
 \le\frac dn.
\end{equation}
Since $d\le D$, Bernstein's inequality~\cite[Theorem~2.9.5]{V26}
gives
\begin{equation}\label{eq:sample-light-tail}
 \Pp\{|L_W(x)|>B\sqrt D\}
 \le2\exp\left\{-n\frac{B^2}{2+4B/3}\right\}.
\end{equation}
Take $B=B_{\mathrm l}:=4[k\log(1+4k)+R+1]$, so that
$B^2/(2+4B/3)\ge k\log(1+4k)+R+1$.
This choice accounts for the size of the product net. Since
$\mathcal T^k$ from Section~\ref{sec:discretization} has at most
$(1+4k)^{kn}$ points, a union bound in
\eqref{eq:sample-light-tail} gives, for $n\ge2$,
\[
 \Pp\left\{\max_{x\in\mathcal T^k}|L_W(x)|>
                  B_{\mathrm l}\sqrt D\right\}
 \le2e^{-(R+1)n}\le n^{-R}.
\]

\subsection{Upper tail and expectation}

\begin{proof}[Proof of \eqref{eq:fw-upper-tail} and the upper bound
in \eqref{eq:fw-expectation}]
Fix $t\ge0$ and take $R=t+2$ in the preceding estimates.
Suppose first that $D<n^2$, and put
$L=1+\lfloor\log_2(1/\tau)\rfloor$.
Since $I_\mu(u)$ is nonincreasing in $\mu$ and $D\ge d$,
Corollary~\ref{cor:multi-local} implies the same discrepancy bound
with reference means $(D/n)|Q|$ and $K=2(R+k+3)$.
Together with the fiber bound, \reviewblue{Proposition~\ref{prop:heavy}} gives,
uniformly over unit vectors,
\[
 \sum_{\mathbf i:\,\prod_j|x_{j,i_j}|>\tau}
 X_{\mathbf i}\prod_j|x_{j,i_j}|
 \le C_{\mathrm h}\sqrt D,
 \qquad C_{\mathrm h}\le C_k(R+1).
\]
\reviewblue{The proposition's explicit constant is linear in}
$\kappa=2(R+k+\lceil\log_2k\rceil)$ and
$K=2(R+k+3)$. The heavy contribution of the mean is at most
\[
 \frac dn\sum_{\mathbf i:\,\prod_j|x_{j,i_j}|>\tau}
                  \prod_j|x_{j,i_j}|
 \le\frac d{n\tau}\sum_{\mathbf i}\prod_jx_{j,i_j}^2
 =\frac d{\sqrt D}\le\sqrt D.
\]
The degree, discrepancy, and light estimates hold simultaneously with
failure probability at most $3n^{-R}$. Adding the light, heavy, and
mean contributions and applying \eqref{eq:net-bound} gives
\[
 \|W\|\le2(B_{\mathrm l}+C_{\mathrm h}+1)\sqrt D
 \le A_k(1+t)\sqrt D,
\]
for a constant $A_k$ independent of $n,m,t$. Here both
$B_{\mathrm l}$ and $C_{\mathrm h}$ grow at most linearly in $R$.

If $D\ge n^2$, then $\tau\ge1$, so every tuple is light.
The light-part union bound and the same net approximation give
the displayed estimate without the degree or discrepancy events.
Thus, in both cases,
\begin{equation}\label{eq:fw-integrable-tail}
 \Pp\{\|X-P\|>A_k(1+t)\sqrt D\}
 \le3n^{-(t+2)}\le n^{-t},\qquad t\ge0.
\end{equation}
Taking $t=r$ proves \eqref{eq:fw-upper-tail}.

Integrating this tail gives an expectation bound uniform in $m$:
\[
 \E\|X-P\|
 \le A_k\sqrt D\left(1+\int_0^\infty n^{-t}\,dt\right)
 =A_k\sqrt D\left(1+\frac1{\log n}\right)
 \le C_k\sqrt D.
\]
Also, $X$ is the sum of $m$ single-entry tensors of norm one, and
$\|P\|=m n^{-k/2}\le m$, so $\|X-P\|\le2m$ deterministically.
Unlike Bernoulli sampling, the fixed-edge model has no empty-tensor
transition when $m\ge1$.
\end{proof}

\subsection{Lower bounds from fiber degrees and conditional fluctuations}

\begin{proof}[Proof of \eqref{eq:fw-lower-tail} and the lower bound
in \eqref{eq:fw-expectation}]
Fix one fiber direction and put $M=n^{k-1}$. Its degrees
$S_1,\ldots,S_M$ have a multinomial law with total $m$ and equal
cell probabilities $1/M$. Write $\Delta=\max_a S_a$ and
$\zeta=\log n/\log(e+\log n/d)$, so $D=\max\{1,d+\zeta\}$.
Since $m\ge1$, deterministically $\Delta\ge\max\{1,d\}$.

Centering may cancel the contribution of a large fiber degree.
We obtain the baseline scale from conditional fluctuations within
fibers and use an unusually large fiber to capture the additional
$\zeta$ term in $D$. We then establish the expectation bound uniformly
in $n$.

\emph{Poissonization.} To handle dependence between the multinomial
fiber degrees, we Poissonize while preserving the required lower-tail
comparison.
Take an infinite sequence of independent edge draws and an independent
$J\sim\operatorname{Poisson}(m/2)$. The degrees $Y_1,\ldots,Y_M$
from the first $J$ draws are independent
$\operatorname{Poisson}(d/2)$ variables. For any event $\mathcal A$
decreasing in the degrees, the coupling and Markov's inequality give
\begin{equation}\label{eq:fw-poissonization}
 \Pp\{\boldsymbol S\in\mathcal A\}\Pp\{J\le m\}
 \le\Pp\{\boldsymbol Y\in\mathcal A\},\qquad
 \Pp\{\boldsymbol S\in\mathcal A\}
 \le2\Pp\{\boldsymbol Y\in\mathcal A\},
\end{equation}
because $\Pp\{J\le m\}\ge1/2$.

\emph{Baseline scale.} Suppose first that $d\ge1$.
Paley--Zygmund applied to $Y_a$ gives
\[
 \Pp\{Y_a\ge d/4\}\ge\frac{d}{4(d+2)}\ge\frac1{12}.
\]
Thus the number of Poissonized degrees at least $d/4$ has mean
at least $M/12$. A binomial lower-tail bound and
\eqref{eq:fw-poissonization} show that
\[
 \Pp\{\#\{a:S_a\ge d/4\}<M/24\}\le2e^{-M/96}.
\]
We now examine fluctuations within each fiber.
Conditional on all degrees, the fiber count vectors are independent.
In a fiber of degree $s$, the counts have a multinomial law with
$s$ trials and equal cell probabilities $1/n$.
For $s\ge1$, let $Z_s$ be this count vector minus
$(s/n)\one$. Projecting the centered fiber onto $\one^\perp$ gives
this vector, so the centered fiber norm is at least
$\|Z_s\|_2$. Expanding the moments of the sum of $s$ independent
vectors $e_I-n^{-1}\one$ gives
\[
 \E\|Z_s\|_2^2=s(1-1/n),\qquad
 \E\|Z_s\|_2^4
 =s^2(1-1/n)^2+\frac{2s(s-1)}n(1-1/n)
 \le3\bigl(\E\|Z_s\|_2^2\bigr)^2.
\]
Paley--Zygmund therefore yields
$\Pp\{\|Z_s\|_2\ge\sqrt s/2\}\ge1/12$.
If at least $M/24$ degrees are at least $d/4$, conditional independence
shows that all corresponding fibers fail this bound with probability
at most $e^{-M/288}$. Testing one successful fiber gives
\begin{equation}\label{eq:fw-dense-lower}
 \Pp\{\|X-P\|<\sqrt d/4\}
 \le2e^{-M/96}+e^{-M/288}\le3e^{-M/288},\qquad d\ge1.
\end{equation}

\emph{Large-fiber scale.} We next capture the contribution of an
unusually large fiber.
Suppose $d\le\log n/8$ and $\zeta\ge16$. Set
$v=\lfloor\zeta/8\rfloor$, so $\zeta/16\le v\le\zeta/8$.
Using $v!\le v^v$, $\zeta\le\log n$, and
$\log_+(v/d)\le\log(e+\log n/d)$, we obtain
\begin{align*}
 \Pp\{Y_a=v\}
 &\ge\exp\{-v\log(2v/d)-d/2\}\\
 &\ge\exp\{-v\log_+(v/d)-v\log2-d/2\}\\
 &\ge\exp\left\{-\frac{3+2\log2}{16}\log n\right\}
 \ge n^{-3/8}.
\end{align*}
Independence of the Poissonized degrees and
\eqref{eq:fw-poissonization} give
\begin{equation}\label{eq:fw-sparse-fiber}
 \Pp\{\Delta<v\}
 \le2(1-n^{-3/8})^M
 \le2e^{-M n^{-3/8}}\le2e^{-n^{5/8}}.
\end{equation}
In this sparse regime, centering cannot erase a large degree:
$q=d/n\le1/4$, and counts are nonnegative integers. Thus a fiber
with count vector $C$ and degree $s$ satisfies
\[
 \|C-q\one\|_2^2
 =\sum_i C_i^2-2qs+nq^2\ge(1-2q)s\ge s/2.
\]
Outside the event in \eqref{eq:fw-sparse-fiber}, we have
$\zeta\le16\Delta$ and $d\le\Delta$, hence $D\le17\Delta$
and $\|X-P\|\ge\sqrt{D/34}$.

\emph{Combining the regimes.} In the remaining cases, $\zeta<16$ or
$d>\log n/8$, we have $D\le17\max\{1,d\}$. For $d\ge1$, apply
\eqref{eq:fw-dense-lower}. For $d<1$, a nonzero entry gives
$\|X-P\|\ge1-d/n\ge1/2$. Consequently, with
$c_*=1/(4\sqrt{17})$,
\[
 \Pp\{\|X-P\|<c_*\sqrt D\}
 \le3e^{-M/288}+2e^{-n^{5/8}}
 \le5e^{-\sqrt n/288},
\]
Choosing $c_0\le\min\{c_*,1/288\}$ and $C_0=5$ proves
\eqref{eq:fw-lower-tail}.

\emph{Expectation.} We need a positive lower bound also for small $n$ and
unbounded $m$. The total squared fluctuation supplies this bound.
Put $N=n^k$
and $F=\|X-P\|_F^2$. Expanding $F$ by pair collisions among the draws
gives
\[
 \E F=m(1-1/N),\qquad
 \E F^2=m^2(1-1/N)^2+\frac{2m(m-1)}N(1-1/N)
 \le\frac53(\E F)^2,
\]
where distinct pair-collision indicators are pairwise independent
and $N\ge4$. Paley--Zygmund gives
$\Pp\{F\ge\E F/2\}\ge3/20$.
The largest fiber norm is at least $\sqrt{F/M}$, so
\[
 \E\|X-P\|\ge\frac3{20}\sqrt{\frac{(1-1/N)d}{2}}
 \ge\frac3{20}\sqrt{3d/8}.
\]
Together with nonemptiness when $d<1$, this gives the baseline lower
bound $c\sqrt{\max\{1,d\}}$ for expectation.
In the sparse case above, \eqref{eq:fw-sparse-fiber} also gives
\[
 \E\|X-P\|\ge(1-2e^{-n^{5/8}})\sqrt{\zeta/32}
 \ge c\sqrt\zeta,
\]
since $1-2e^{-2^{5/8}}>0$. The larger of these bounds controls
$\sqrt D$; in the other cases $D\le17\max\{1,d\}$.
This proves the expected lower bound in \eqref{eq:fw-expectation}.
\end{proof}

\section[Light-fitting regularization]{\reviewblue{Light-fitting regularization}}
\label{sec:minimax}

Throughout this section, we use the model and parameter space from
Section~\ref{sec:main}:
$k\ge2$, $n\ge16$, and the known bound satisfies $\reviewblue{0\le p\le1}$.
We prove the upper bound by analyzing the light-fitting regularization
algorithm. The matching lower bound follows from
Zhou and Zhu~\cite[Theorem~2.7]{ZZ21}.

\subsection{Analysis of the light-fitting estimator}

\reviewblue{Use the estimator in Algorithm~\ref{alg:light-fitting}.}

\begin{proof}[Proof of the upper bound in Theorem~\ref{thm:minimax}]

If $pn^{k/2}\le\sqrt{np}$, then $\widehat P=0$ and
\[
 \|\widehat P-P\|\le\|P\|_F\le pn^{k/2},
\]
which gives the required bound. This includes $p=0$.
Assume henceforth that $pn^{k/2}>\sqrt{np}$, so $p>0$ and
$\widehat P$ is the fit in \eqref{eq:mm-fit}.
Recall that $\tau=\sqrt{p/n}=\sqrt{np}/n$ and
\[
 \mathcal L(x)=\left\{\mathbf i:
                   \prod_{j=1}^k|x_{j,i_j}|\le\tau\right\}.
\]
The objective $F_T(A)$ is the maximum absolute light residual of
$T-A$ over the product net $x\in\mathcal T^k$.

We first obtain a deterministic error bound. The fit controls the
light tuples, while the constraint $A\in[0,p]^{[n]^k}$ controls the
error on heavy tuples. Since $P$ is feasible,
$F_T(\reviewblue{\widehat P})\le F_T(P)$, and for every $x\in\mathcal T^k$,
\[
 \left|\sum_{\mathbf i\in\mathcal L(x)}
    (\reviewblue{\widehat p_{\mathbf i}}-\reviewblue{p_{\mathbf i}})
                    \reviewblue{\prod_{j=1}^k x_{j,i_j}}\right|
 \le 2F_T(P).
\]
For the heavy tuples, the entrywise constraint gives
$|\reviewblue{\widehat p_{\mathbf i}}-\reviewblue{p_{\mathbf i}}|\le p$.
Together with \reviewblue{$\sum_{\mathbf i}\prod_{j=1}^k x_{j,i_j}^2=1$},
this yields

\begin{align*}
 \sum_{\mathbf i\notin\mathcal L(x)}
  |\reviewblue{\widehat p_{\mathbf i}}-\reviewblue{p_{\mathbf i}}|
                         \bigl|\prod_{j=1}^k x_{j,i_j}\bigr|
 &\le p\sum_{\mathbf i\notin\mathcal L(x)}
                         \bigl|\prod_{j=1}^k x_{j,i_j}\bigr|\\
 &\le \frac p\tau\sum_{\mathbf i}\prod_{j=1}^k x_{j,i_j}^2
 =\sqrt{np}.
\end{align*}

Combining the two estimates and applying the net bound
\eqref{eq:net-bound} gives
\begin{equation}\label{eq:mm-deterministic}
 \|\reviewblue{\widehat P}-P\|
               \le4F_T(P)+2\sqrt{np}.
\end{equation}
No independence of the fitted entries is needed for this
deterministic estimate.

It remains to bound the expected light residual at the true mean.
For fixed $x$, let $Y_x$ denote its light residual at $P$.
The summands are centered, bounded by $\tau$, and have total variance
at most $p$.
Bernstein's exponential-moment bound gives, for $0<\lambda<3/\tau$,
\[
 \log\E_P e^{\pm\lambda Y_x}
 \le\frac{\lambda^2p}{2(1-\lambda\tau/3)}.
\]
To pass from an individual light residual to the maximum over the
net, use
$e^{\lambda|Y_x|}\le e^{\lambda Y_x}+e^{-\lambda Y_x}$ and bound
the exponential of the maximum by the sum of exponentials.
Together with $|\mathcal T|^k\le(1+4k)^{kn}$ and Jensen's inequality,
this gives
\[
 \E_P F_T(P)
 \le\frac{\log2+kn\log(1+4k)}{\lambda}
       +\frac{\lambda p}{2(1-\lambda\tau/3)}.
\]
Taking $\lambda=1/\tau$ gives, uniformly in $P\in\mathcal P(p)$,
\begin{equation}\label{eq:mm-light-mean}
 \E_P F_T(P)
 \le\bigl[k\log(1+4k)+\log2+3/4\bigr]\sqrt{np}.
\end{equation}

Combining \eqref{eq:mm-deterministic} and \eqref{eq:mm-light-mean}
gives the $C_k\sqrt{np}$ bound in the fitted branch. Together
with the zero-estimator bound above, this proves the result.
\end{proof}

\subsection[Minimax lower bound for spectral-norm estimation]{\reviewblue{Minimax lower bound for spectral-norm estimation}}

For every $0<p\le1$, Zhou and Zhu~\cite[Theorem~2.7]{ZZ21}
give the minimax probability lower bound for the same independent-entry
model and parameter space $\mathcal P(p)$, with no sparsity restriction.

To obtain the
expected-risk lower bound in \eqref{eq:mm-rate}, apply
$\E Z\ge a\Pp\{Z\ge a\}$ for every nonnegative $Z$ and $a>0$.

\subsection*{Acknowledgements}

GPT-6 assisted in developing the proof of the heavy-tuple part, improving a loose estimate identified in the authors' earlier work \cite{ZZ21}. The authors independently verified the
argument and take full responsibility for the final proof. Y.Z. was partially supported by the Simons Grant MPS-TSM-00013944 and NSF DMS-2606337. This work was carried out while Y.Z. was visiting the Simons Institute for the Theory of Computing during the Spectral Theory Beyond Graphs program in Fall 2026.

\bibliographystyle{plain}
\bibliography{ref}
\end{document}